\documentclass[11pt]{amsart}
\usepackage{adjustbox}
\usepackage{amssymb}
\usepackage{amsmath}
\usepackage{fancyhdr}
\usepackage[british]{babel}
\usepackage{geometry,mathtools}
\usepackage{enumitem}
\usepackage{algpseudocode}
\usepackage{dsfont}
\usepackage{centernot}
\usepackage{xstring}
\usepackage{colortbl}
\usepackage[section]{algorithm}
\usepackage{graphicx}
\usepackage[font={footnotesize}]{caption}
\usepackage[usenames,dvipsnames,table]{xcolor}
\usepackage{tikz}
\usepackage[h]{esvect}
\usepackage[
		bookmarksopen=true,
		bookmarksopenlevel=1,
		colorlinks=true,
		linkcolor=darkblue,
        linktoc=page,
		citecolor=darkblue,
]{hyperref}
\usepackage{bbm}
\usepackage{mathrsfs}
\usepackage{hyphenat}
\usepackage{soul}
\usepackage{cancel}
\usetikzlibrary{arrows.meta, decorations.pathmorphing, decorations, fit}

\allowdisplaybreaks

\title[Counting tight Hamilton cycles in Dirac hypergraphs]{Counting tight Hamilton cycles in Dirac hypergraphs}
\author[F.~Joos]{Felix Joos}
\address[F.~Joos]{Institut f\"ur Informatik, Universit\"at Heidelberg, Heidelberg,
Germany
}
\email{joos@informatik.uni-heidelberg.de}

\author[X.~Xie]{Xinyue Xie}
\address[X.~Xie]{Institut f\"ur Informatik, Universit\"at Heidelberg, Heidelberg,
Germany
}
\email{xie@informatik.uni-heidelberg.de}

\date{\today}

\thanks{The research leading to these results was supported by the Deutsche Forschungsgemeinschaft (DFG, German Research Foundation) -- 428212407}

\newtheorem{theorem}[algorithm]{Theorem}

\newtheorem{lemma}[algorithm]{Lemma}
\newtheorem{cor}[algorithm]{Corollary}
\newtheorem{fact}[algorithm]{Fact}

\theoremstyle{definition}

\newtheoremstyle{claimstyle}{5pt}{5pt}{\em}{5pt}{\em}{:}{5pt}{}
\theoremstyle{claimstyle}

\newtheoremstyle{stepstyle}{10pt}{5pt}{\em}{0pt}{\em}{:}{5pt}{}
\theoremstyle{stepstyle}

\numberwithin{equation}{section}

\definecolor{darkblue}{rgb}{0,0,0.5}

\newdimen\margin
\def\textno#1&#2\par{
   \margin=\hsize
   \advance\margin by -4\parindent
          \setbox1=\hbox{\sl#1}
   \ifdim\wd1 < \margin
      $$\box1\eqno#2$$
   \else
      \bigbreak
      \hbox to \hsize{\indent$\vcenter{\advance\hsize by -3\parindent
      \it\noindent#1}\hfil#2$}
      \bigbreak
   \fi}

\begin{document}

\newcommand{\new}[1]{\textcolor{red}{#1}}
\newcommand{\old}[1]{\textcolor{red}{\st{#1}}}
\newcommand{\oldblock}[1]{\begin{tikzpicture}\node[inner sep=0pt, outer sep=0pt] at (0,0) (bb) {\begin{minipage}{\textwidth}\color{red}#1\end{minipage}};\draw[-, red] (bb.north west) -- (bb.south east);\end{tikzpicture}}
\def\COMMENT#1{}
\def\TASK#1{}

\newcommand{\todo}[1]{\begin{center}\color{red}\textbf{to do:} #1 \end{center}}

\def\op{\operatorname}
\def\eps{\varepsilon}

\newcommand{\pr}{\mathbb{P}}
\newcommand{\ex}{\mathbb{E}}

\newcommand{\bN}{\mathbb{N}}
\newcommand{\bZ}{\mathbb{Z}}
\newcommand{\bQ}{\mathbb{Q}}
\newcommand{\bR}{\mathbb{R}}
\newcommand{\bC}{\mathbb{C}}

\newcommand{\bw}{\mathbf{w}}
\newcommand{\bx}{\mathbf{x}}
\newcommand{\by}{\mathbf{y}}
\newcommand{\bz}{\mathbf{z}}

\newcommand{\cE}{\mathcal{E}}
\newcommand{\cF}{\mathcal{F}}
\newcommand{\cG}{\mathcal{G}}
\newcommand{\cS}{\mathcal{S}}
\newcommand{\cT}{\mathcal{T}}
\newcommand{\cQ}{\mathcal{Q}}
\newcommand{\cP}{\mathcal{P}}
\newcommand{\cA}{\mathcal{A}}
\newcommand{\cU}{\mathcal{U}}
\newcommand{\cZ}{\mathcal{Z}}

\newcommand{\1}{\mathbbm 1}
\newcommand{\supp}{{\rm supp}}
\newcommand{\diff}{\mathop{}\!\mathrm{d}}

\newcommand\restrict[1]{\raisebox{-.5ex}{$|$}_{#1}}

\newcommand{\Set}[1]{\{#1\}}
\newcommand{\set}[2]{\{#1\,:\;#2\}}

\newcommand{\norm}[1]{\|#1\|}

\newcommand{\AU}[1]{\op{Aut}(#1)}

\begin{abstract} 
Suppose $G$ is a $k$-uniform hypergraph on $n$ vertices such that every $(k-1)$-subset~$S$ of $V(G)$ belongs to at least $\delta n$ edges, where $\delta> 1/2$. Let $\Psi(G)$ denote the number of tight Hamilton cycles in $G$, that is, cyclic orderings of $V(G)$ in which every $k$ consecutive vertices form an edge. We prove that $\log\Psi(G)\ge kh(G)-n\log{n\choose k-1}+n\log n-n\log e-o(n)$, where $h(G)$ is the hypergraph entropy of $G$, defined via perfect fractional matchings. 
This bound is tight, for example, for all (nearly) regular hypergraphs, in particular for the binomial random hypergraph. 
It also implies a conjecture by Ferber, Hardiman and Mond, stating that $\Psi(G)\ge (\delta-o(1))^n n!$.
\end{abstract}

\maketitle

\section{Introduction}

A classical theorem of Dirac~\cite{Dirac:52} states that every graph $G$ on $n\ge 3$ vertices with minimum degree $\delta(G)\ge n/2$ contains a Hamilton cycle. In fact, once this degree threshold is reached, the graph typically contains many Hamilton cycles. This observation motivated a line of research focused on counting Hamilton cycles in graphs and on extending such results to hypergraphs.

We denote the number of Hamilton cycles in a graph $G$ by $\Psi(G)$ and the number of perfect matchings by $\Phi(G)$. In a seminal work, Cuckler and Kahn~\cite{CK:09:ub,CK:09:lb} established asymptotically tight bounds on $\Psi(G)$ and $\Phi(G)$ in terms of a suitably defined notion of graph entropy.

To this end, for a $k$-uniform graph ($k$-graph) $G$ on $n$ vertices, let a \emph{perfect fractional matching~$\bx$} be a function $\bx:E(G)\to \bR_{\ge 0}$ such that $\sum_{v\in e\in E(G)}\bx(e)=1$ for every vertex $v\in V(G)$. The \emph{entropy of $\bx$} is
$$h(\bx)\coloneqq \sum_{e\in E(G)}\bx(e)\log \frac{1}{\bx(e)}.$$
The \emph{entropy of the $k$-graph $G$} is then \mbox{$h(G)\coloneqq \sup\{h(\bx)\colon\bx\text{ is a perfect fractional matching of }G\}$}.

\begin{theorem}[{\cite{CK:09:ub,CK:09:lb}}]\label{thm: CK bounds}
Let $G$ be a graph on $n$ vertices with $\delta(G)\ge n/2$. Then
$$\log \Psi(G)=2h(G)-n\log e-o(n).$$
If $n$ is even, then
$$\log \Phi(G)=h(G)-\frac{n}{2}\log e-o(n).$$
\end{theorem}

In the same work, Cuckler and Kahn also proved a tight lower bound on $h(G)$ in terms of the minimum degree. Combined with the above theorem, this yields the following result.

\begin{theorem}[{\cite{CK:09:lb}}]\label{thm: CK lower bound}
Let $G$ be a graph on $n$ vertices with $\delta(G)\ge\delta n\ge n/2$. Then 
$$\Psi(G)\ge \left(\delta+o(1)\right)^{n}n!.$$
\end{theorem}

Observe that this is asymptotically tight for the binomial random graph. These entropy-based results for graphs naturally raise the question of how hamiltonicity and its enumeration behave in the richer setting of hypergraphs. We begin by carefully defining the relevant notions of cycles in hypergraphs. Let $\ell$ be an integer such that $0\le \ell\le k-1$. A $k$-graph $C$ is an \emph{$\ell$-cycle} if there exists a cyclic ordering of its vertex set $V(C)$ such that every edge of $C$ consists of $k$ consecutive vertices, and every pair of consecutive edges intersects in precisely $\ell$ vertices. A \emph{Hamilton $\ell$-cycle} in a $k$-graph $G$ is an $\ell$-cycle $C$ with $V(C)=V(G)$. A $k$-graph $G$ is \emph{hamiltonian} if it contains a Hamilton $(k-1)$-cycle, which is also referred to as a \emph{tight Hamilton cycle}. A \emph{perfect matching} of~$G$ is a Hamilton 0-cycle where the vertex ordering is ignored; equivalently, it is a set of pairwise disjoint edges that covers the vertex set of $G$.

R\"odl, Ruci\'nski and Szemer\'edi~\cite{RRS:06,RRS:08} generalized Dirac's theorem by proving an asymptotic version for $k$-graphs, thereby establishing a threshold for the existence of Hamilton $\ell$-cycles.

\begin{theorem}[{\cite{RRS:06,RRS:08}}]
Let $k\ge3$, $\delta>1/2$, and let $G$ be a $k$-graph on $n$ vertices, where $n$ is sufficiently large. If the minimum $(k-1)$-degree satisfies $\delta_{k-1}(G)\ge\delta n$, then $G$ is hamiltonian.
\end{theorem}

As in the graph setting, we use $\Psi(G)$ and $\Phi(G)$ to denote the number of tight Hamilton cycles and perfect matchings, respectively, in a $k$-graph $G$. Let $K_n^{(k)}$ denote the complete $k$-graph on $n$ vertices. Recently, Kwan, Safavi, and Wang~\cite{KSW:24} generalized the lower bounds on the number of perfect matchings in Theorems~\ref{thm: CK bounds} and~\ref{thm: CK lower bound} from graphs to hypergraphs.

\begin{theorem}[{\cite{KSW:24}}]\label{thm: KSW}
Fix $k\ge 2$ and $\delta>1/2$. Let $G$ be an $n$-vertex $k$-graph with $\delta_{k-1}(G)\ge \delta n$ and $k$ divides $n$. Then
$$\log \Phi(G)=h(G)-\frac{k-1}{k}n\log e+o(n)\quad\text{and}\quad\Phi(G)\ge \Phi(K^{(k)}_n)\cdot(\delta+o(1))^{n/k}.$$
\end{theorem}

Let $\Psi_{\ell}(G)$ denote the number of Hamilton $\ell$-cycles in a $k$-graph $G$. Observe that the expected number of Hamilton $\ell$-cycles in a random $k$-graph on $n$ vertices, in which every edge appears independently with probability $p$, is \mbox{$\Psi_{\ell}(K_n^{(k)})\cdot p^\frac{n}{k-\ell}$}. In particular, for tight Hamilton cycles we have $\Psi(K_n^{(k)})=\frac{n!}{2n}$. For $0\le \ell\le k-2$, Ferber, Krivelevich and Sudakov~\cite{FKS:16}, as well as Ferber, Hardiman and Mond~\cite{FHM:23}, showed that this estimate using binomial random $k$-graphs in fact provides a correct lower bound once $p$ is replaced by a parameter depending on the minimum $(k-1)$-degree. This extends Theorem~\ref{thm: CK lower bound} to hypergraphs.

\begin{theorem}[{\cite{FHM:23, FKS:16}}]\label{thm: FHM}
Let $0\le \ell\le k-2$ and $\delta>1/2$. Let $n$ be sufficiently large and divisible by $k-\ell$. Let $G$ be a $k$-graph on $n$-vertices with $\delta_{k-1}(G)\ge\delta n$. Then the number of Hamilton $\ell$-cycles in $G$ is at least $(1-o(1))^n\cdot\Psi_{\ell}(K_n^{(k)})\cdot\delta^{\frac{n}{k-\ell}}$.
\end{theorem}

Futhermore, Ferber, Hardiman and Mond~\cite{FHM:23} conjectured that this result extends to $\ell=k-1$, that is, to tight Hamilton cycles. This is a stronger form of Conjecture 7.1 by Glock, Gould, Joos, K\"uhn and Osthus~\cite{GGJKO:21}, where the authors also asked whether the entropy approach of Cuckler and Kahn can be generalized to hypergraphs.

In this paper, we answer this question affirmatively by establishing a tight lower bound for~$\Psi(G)$ in terms of hypergraph entropy. We prove the following two main theorems. Theorem~\ref{main theorem - with entropy} generalizes the lower bound on $\Psi(G)$ in Theorem~\ref{thm: CK bounds} to hypergraphs. Theorem~\ref{main theorem - without entropy} follows as a corollary to Theorem~\ref{main theorem - with entropy} and resolves the above-mentioned conjectures.

\begin{theorem}\label{main theorem - with entropy}
Suppose $k\ge2$ and $\delta> 1/2$. Let $G$ be a $k$-graph on $n$ vertices with \mbox{$\delta_{k-1}(G)\ge\delta n$}. Then 
$$\log\Psi(G)\ge kh(G)-n\log{n\choose k-1}+n\log n-n\log e-o(n).$$
\end{theorem}

\begin{theorem}\label{main theorem - without entropy}
Suppose $k\ge2$ and $\delta>1/2$. Let $G$ be a $k$-graph on $n$ vertices with \mbox{$\delta_{k-1}(G)\ge\delta n$}. Then $G$ contains at least $(\delta-o(1))^n n!$ tight Hamilton cycles.
\end{theorem}

The method presented in this paper also applies in the more general setting of counting Hamilton $\ell$-cycles in a \mbox{$k$-graph}~$G$ on $n$ vertices with $\delta_{k-1}(G)\ge \delta n$ for some $\delta>1/2$. We define
$$C_{n,k,\ell}\coloneqq \begin{cases}
(\pi_1!\pi_2!)^{\frac{n}{k-\ell}}\quad&\text{if }\ell\in[k-1]\\
(n/k)!(k!)^{\frac{n}{k}}\quad&\text{if }\ell=0
\end{cases},$$
where $\pi_1=k\bmod(k-\ell)$ and $\pi_2=(k-\ell)-\pi_1$. Observe that when $(k-\ell)|n$, we have \mbox{$\Psi_{\ell}(K_n^{(k)})=\frac{n!}{\frac{2n}{k-\ell}C_{n,k,\ell}}$} for $\ell\in[k-1]$, and $\Phi(K_n^{(k)})=\frac{n!}{C_{n,k,0}}$. Here we state the general result in Theorem~\ref{main theorem - with entropy - non-tight}. The proofs carry over with only minor modifications, and we therefore defer the details to the appendix. The lower bound in Theorem~\ref{thm: KSW} and Theorem~\ref{thm: FHM} follow as corollaries when $\ell=0$ and $0\le \ell\le k-2$, respectively.

\begin{theorem}\label{main theorem - with entropy - non-tight}
Suppose $k\ge2$, $\ell\in[k-1]_0$ and $\delta> 1/2$. Let $G$ be a $k$-graph on $n$ vertices with $(k-\ell)|n$ and \mbox{$\delta_{k-1}(G)\ge\delta n$}. Then
$$\log\Psi_{\ell}(G)\ge \frac{k}{k-\ell}h(G)-\frac{n}{k-\ell}\log{n\choose k-1}\frac{(n-k+1)!}{(n-\ell)!n}-n\log e-\log C_{n,k,\ell}-o(n).$$
\end{theorem}

The paper is organized as follows. In Section~\ref{section: Proof overview}, we provide an outline of our proof method. Section~\ref{section: Preliminaries} introduces the necessary notations and definitions and proves several basic results. In Section~\ref{section: Subgraph entropy}, we study how hypergraph entropy behaves under the removal of vertex subsets. In Section~\ref{section: Probabilistic properties of random walks}, we study random walks in $k$-graphs and show that they are almost always well-behaved. In Section~\ref{section: Counting Hamilton cycles}, we put together the results developed earlier to count long paths and extend them to Hamilton cycles, thereby proving our main theorems.

\section{Proof overview}\label{section: Proof overview}
We provide an outline for the proof of Theorem~\ref{main theorem - with entropy}. Let $k\ge 2$, $\delta> 1/2$ and $G$ be a $k$-graph on $n$ vertices with $\delta_{k-1}(G)\ge \delta n$. 
We first reserve a set $U\subseteq V(G)$ of size~$n^{\alpha}$, with $\alpha$ close to 1, such that every $(k-1)$-subset of $V(G)$ has many neighbours in $U$. Let $G'=G-U$ and $n'\coloneqq |V(G')|$. Observe that $\delta_{k-1}(G')\ge \delta'n'$ for some $\delta'>1/2$. We construct long tight paths in $G'$. The strong connectivity of $U$ then allows us to complete these paths into Hamilton cycles of $G$.

In $G'$, we take a perfect fractional matching $\bx$ whose entropy $h(\bx)$ is close to $h(G)$. We use $\bx$ as a probability distribution that guides the construction of a tight path. Moreover, we require $\bx$ to be $b$-normal, which means that the edge weights are roughly evenly distributed so that along a long path, the probability of visiting any vertex is roughly uniform.

Let $\vec\cS(G')$ denote the set of all $(k-1)$-tuples of vertices of $G'$. Since extensions of a tight path are determined by its last $k-1$ vertices, we study a random walk $Z$ on $\vec\cS(G)$, viewed as a Markov chain whose transition matrix is defined in terms of $\bx$. Starting at some $\vec S\in\vec\cS(G)$, we run this Markov chain for $m'\coloneqq (n')^{1/3}$ steps. The length $m'$ is chosen large enough so that the rapid mixing property of a Markov chain applies. It is also small enough so that with high probability, the walk does not revisit vertices and thus yields a tight path $P$ in $G'$.

We show that with high probability, the random walk $Z$ is well-behaved, which in particular implies the following. First, $Z$ yields a tight path $P$ in $G$ from $\vec S$ to some $\vec T\in\vec\cS(G')$. Second, the subgraph $G''=G'-(P-T)$ with $n''\coloneqq V(G'')$ still satisfies a minimum codegree condition $\delta_{k-1}(G'')\ge \delta'' n''$ for some $1/2<\delta''\le \delta'$. Lastly, $G''$ admits a $b'$-normal perfect fractional matching $\bx'$, with $b'$ only slightly larger than $b$, such that $h(\bx')$ scales appropriately with $n''$ and $h(\bx')$ is close to $h(G'')$. Hence we can treat $G''$ in the same way as $G'$ and construct random walks starting at $\vec T$.

We iterate this procedure. Starting from $G_0\coloneqq G'$ and some $\vec S_0\in\vec\cS(G_0)$, we remove a well-behaved path $P_i$ of length $|V(G_i)|^{1/3}$, from $\vec S_i$ to some $\vec S_{i+1}\in\vec\cS(G_i)$, to obtain the next subgraph $G_{i+1}\coloneqq G_i-(P_i-S_{i+1})$. At each step, let $z_i$ denote the number of well-behaved paths, which can be lower bounded using the entropy of the random walk. We continue this process until $Q\coloneqq P_0\cup...\cup P_{\kappa}$ has length $|Q|\ge n'-n'^{1-\beta}$ for some $\kappa\in\mathbb{N}$ and small $\beta>0$. 

Finally, the set $U$ is used to absorb the remaining vertices and complete $Q$ into a tight Hamilton cycle in $G$. The product $\prod_{i=0}^{\kappa}z_i$ provides the desired lower bound on $\Psi(G)$.

\section{Preliminaries}\label{section: Preliminaries}

In this section, we introduce some notations and definitions that will be used throughout this paper, and establish several preliminary results.

\subsection{Notation}
For any integer $n\ge 1$, we define $[n]\coloneqq \{1,...,n\}$ and $[n]_0\coloneqq \{0,...,n\}$. For real numbers $a, b$ and some error term $c\ge 0$, we write $a=b\pm c$ if $b-c\le a\le b+c$. Expressions with more error terms are defined similarly and should be read from inner to outer brackets. For $\alpha,\beta\in(0, 1]$, we write that a statement holds for $\alpha\ll\beta$ if there is a non-decreasing function $f:(0, 1]\to (0,1]$ such that the statement holds for any $\alpha\in(0,f(\beta)]$. Hierarchies with more constants are defined similarly and should be read from right to left. Note that if $1/k$ and $1/n$ appear in a hierarchy, this implicitly means that $k$ and $n$ are natural numbers, respectively, and in addition that $k\ge 2$. We ignore rounding issues when they do not affect the argument.

We use standard asymptotic notation. For functions $f=f(n)$ and $g=g(n)$, we write $f=O(g)$ if there exists a constant $C>0$ such that $f(n)\le Cg(n)$ for all $n$, $f=o(g)$ if $f/g\to 0$ as $n\to \infty$, and $f=\Theta(g)$ if $f=O(g)$ and $g=O(f)$.

We also use standard probabilistic notation. We denote the probability of an event $\cE$ by $\pr[\cE]$ and the expected value of a random variable $X$ by $\ex[X]$.

We write $\log\coloneqq\log_2$ and $\ln\coloneqq\log_e$. We use the convention $0\log 0=0$. For $x\ge 0$, we often use the inequality $\log (1+x)\le 2x$, which follows from $1+x\le e^x\le 2^{2x}$. We also often use $\frac{n^k}{k^k}\le {n\choose k}\le \frac{n^k}{k!}$ to bound binomial coefficients. Moreover, we use binomial approximation $(1+x)^{\alpha}=1+\alpha x\pm O(x^2)=1\pm O(x)$ when $|x|<1$ and $|\alpha x|$ is sufficiently small. We use Stirling's approximation $\log(n!)=n\log\frac{n}{e}\pm\log n$.

\subsection{Dirac hypergraphs} We introduce the basic graph theoretic definitions and notation.

A \emph{\mbox{$k$-uniform} hypergraph} (or \mbox{\emph{$k$-graph}}) $G$ is a hypergraph in which each edge is a set of exactly $k$ vertices. We denote its vertex set by $V(G)$ and its edge set by $E(G)$. For any integer $d$ with $d\in[k-1]_0$, the \emph{neighbourhood} of a \mbox{$d$-subset} $S\subseteq V(G)$ in $G$ is defined by \mbox{$N_G(S)\coloneqq \{T\subseteq V(G)|S\cup T\in E(G)\}$}, and the \emph{degree} of $S$ in $G$ is $\deg_G(S)\coloneqq |N_G(S)|$. Note that when $d=0$, we have $N_G(\emptyset)=E(G)$. When $d=1$~or~$d=k-1$, we simply write the singleton sets as vertices. That is, \mbox{$N_G(v)\coloneqq \{T\subseteq V(G)|\{v\}\cup T\in E(G)\}$} for each vertex $v\in V(G)$ and \mbox{$N_G(S)\coloneqq \{v\in V(G)|S\cup\{v\}\in E(G)\}$} for each \mbox{$(k-1)$-subset} $S\subseteq V(G)$. The \emph{minimum $d$-degree}~$\delta_d(G)$ of $G$ is the minimum degree over all $d$-subsets of $V(G)$. The minimum $d$-degree of a $k$-graph is monotone in the following sense.

\begin{fact}\label{lemma: monotonicity of minimum degrees}
Let $G$ be a $k$-graph. Then
$$\frac{\delta_0(G)}{{n\choose k}}\ge\frac{\delta_1(G)}{{n-1 \choose k-1}}\ge ...\ge \frac{\delta_{k-1}(G)}{{n-k+1\choose 1}}.$$
\end{fact}

Let $G$ be a $k$-graph on $n$ vertices. Let $1/2< \delta<1$ be a constant. We say that $G$ is \emph{$\delta$-Dirac} if $\delta_{k-1}(G)\ge\delta n$. We say that $G$ is an \mbox{\emph{$(n, k, \delta)$-graph}} if $G$ is a $\delta$-Dirac $k$-graph on $n$ vertices. We use $\hat{\delta}\coloneqq \delta-1/2$ as a shorthand to denote the offset of $\delta$ above the $1/2$ threshold.

Suppose $G$ is a hypergraph, $H$ is a subgraph of $G$ and $U\subseteq V(G)$. We write $G[U]$ for the subgraph of $G$ induced by $U$. We define $G-U\coloneqq G[V(G)\setminus U]$ and $G-H\coloneqq G-V(H)$.

Let $G$ be a $k$-graph. We say a subgraph~$W$ of $G$ is a \emph{(tight) walk} of length $L$ if $W$ admits an ordering of its vertices $V(W)=(v_1, ...v_{L+k-1})$ such that $E(W)=\{\{v_i,... v_{i+k-1}\}|i\in [L]\}$. The walk~$W$ is a \emph{(tight) path} if the vertices $v_1,...,v_{L+k-1}$ are distinct. The walk~$W$ is a \emph{(tight) cycle} if $(v_1,...v_{k-1})=(v_{L+1},...,v_{L+k-1})$ and the vertices $v_1,...,v_L$ are distinct. A \emph{(tight) Hamilton path} in $G$ is a tight path~$P$ such that $V(P)=V(G)$. A \emph{(tight) Hamilton cycle} in $G$ is a tight cycle~$C$ such that $V(C)=V(G)$.

Since we focus exclusively on tight walks, paths, and cycles in this paper, we will henceforth omit the word ``tight".

\subsection{Hypergraph entropy}
We define perfect fractional matchings and hypergraph entropy.

Let $G$ be a $k$-graph on $n$ vertices. An \emph{edge weighting} of $G$ is a function $\bx:E(G)\to \bR_{\ge 0}$. For all $e\in E(G)$, we write $\bx_e$ or $\bx[e]$ for $\bx(e)$. We call $\bx$ a $\emph{perfect fractional matching}$ if $\sum_{v\in e\in E(G)}\bx_e=1$ for every vertex $v\in V(G)$. Let $F\subseteq E(G)$, we define the \emph{entropy of $F$ induced by $\bx$} as
$$h_{\bx}(F)\coloneqq \sum_{e\in F}\bx_e\log \frac{1}{\bx_e}.$$
We define the \emph{entropy of $\bx$} by \mbox{$h(\bx)\coloneqq h_{\bx}(E(G))$}. We define the \emph{entropy} of the hypergraph $G$ by \mbox{$h(G)\coloneqq \sup\{h(\bx)|\bx\text{ is a perfect fractional matching of }G\}$}.

Let $G$ be a $k$-graph on $n$ vertices, $\bx$ be an edge weighting of $G$ and $b\ge 1$. We say $\bx$ is \emph{$b$-normal} if for all $e\in E(G)$, we have
$$\frac{1}{bn^{k-1}}\le \bx_e\le \frac{b}{n^{k-1}}.$$

As $h(G)$ is monotone increasing with respect to taking supergraphs, it follows that \mbox{$h(G)\le h(K_n^{(k)})\le\frac{k-1}{k}n\log n$}, where the second inequality is an immediate consequence of the definition of graph entropy and Jensen's inequality applied to the concave function $x\mapsto \log x$. A tight lower bound for the entropy of an $(n,k,\delta)$-graph is proved in \cite{KSW:24}.

\begin{theorem}[{\cite[Theorem~6.1]{KSW:24}}]\label{thm: graph entropy bound}
Let $G$ be an $(n,k,\delta)$-graph. Then,
$$\frac{n}{k}\log{n\choose k-1}+\frac{n}{k}\log \delta\le h(G)\le\frac{k-1}{k}n\log n.$$
\end{theorem}

It then follows that $h(G)=\Theta(n\log n)$ for $(n,k,\delta)$-graphs. In fact, the same asymptotic bound holds for any $b$-normal perfect fractional matching of $G$.

\begin{lemma}[{\cite[Lemma~2.6]{KSW:24}}]\label{lemma: pfm entropy bound}
Let $G$ be a $k$-graph on $n$ vertices and $\bx$ be a perfect fractional matching of $G$ with $\bx_e\le L$ for every edge $e\in E(G)$. Then
$$h(\bx)\ge\frac{n}{k}\log\frac{n}{L^2 k|E(G)|}.$$
\end{lemma}

\begin{cor}\label{cor: pfm entropy bound}
Let $G$ be a $k$-graph on $n$ vertices and $\bx$ be a $b$-normal perfect fractional matching of $G$ for some $b\ge 1$. Then,
$h(\bx)=\frac{k-1}{k}n\log n-O_b(n)$.
\end{cor}

\begin{proof}
By Theorem~\ref{thm: graph entropy bound}, $h(\bx)\le h(G)\le\frac{k-1}{k}n\log n$. By Lemma~\ref{lemma: pfm entropy bound}, we have
$$h(\bx)\ge\frac{n}{k}\log\frac{n}{(b/n^{k-1})^2 k|E(G)|}\ge\frac{n}{k}\log\frac{n^{2k-1}}{b^2 k{n\choose k}}\ge\frac{n}{k}\log\frac{n^{k-1}}{b^2}\ge \frac{k-1}{k}n\log n-O_b(n),$$
which completes the proof.
\end{proof}

The following lemma states that every $(n,k,\delta)$-graph admits a $b$-normal perfect fractional matching whose entropy is close to the entropy of the graph. For digraphs, this is proven as Lemma 4.6 in~\cite{JS:24}.

\begin{lemma}[{\cite[Lemma~3.1]{KSW:24}}]\label{lemma: existence of b-normal pfm}
Suppose $1/b\ll\eps\ll\hat{\delta},1/k$. Let $G$ be an $(n,k,\delta)$-graph with $n$ sufficiently large with respect to $\hat{\delta}$ and $k$. Then there is a $b$-normal perfect fractional matching $\bx$ of $G$ such that $h(\bx)\ge h(G)-\eps n$.
\end{lemma}

\subsection{Information-theoretic entropy}
We recall the Shannon entropy for random variables. Let $X$ be a random variable with finite support. Then the \emph{entropy} of $X$ is defined as
$$H(X)\coloneqq\sum_{x\in\supp(X)}\pr[X=x]\log\frac{1}{\pr[X=x]}.$$
For an event $\cE$ on the same probability space as $X$, we write $H(X|\cE)$ for the entropy of $X$ in the conditional probability space given $\cE$. For another random variable $Y$ with finite support on the same probability space as $X$, the \emph{conditional entropy} of $X$ given $Y$ is defined as
$$H(X|Y)\coloneqq \sum_{y\in \supp(Y)}\pr[Y=y]H(X|{Y=y}).$$

We need the following facts on entropy. Fact \ref{fact: max entropy} relates entropy to the size of the support. Fact \ref{fact: CR for entropy} is the chain rule for entropy, allowing us to deal with a sequence of random variables on the same probability space. Fact~\ref{fact: entropy given a likely event} states that the entropy does not decrease significantly when conditioned on a likely event.

\begin{fact}\label{fact: max entropy}
Let $X$ be a random variable with finite support. Then $H(X)\le \log|\supp(X)|$, where equality holds when $X$ is a uniform distribution.
\end{fact}

\begin{fact}\label{fact: CR for entropy}
Let $X_1,...,X_m$ be random variables on the same probability space. Then
$$H(X_1,...,X_m)=\sum_{i=1}^m H(X_i|X_1,...,X_{i-1}).$$
\end{fact}

\begin{fact}[{\cite[Lemma~3.1]{JS:24}}]\label{fact: entropy given a likely event}
Let $X$ be a random variable with values in $\{x_1,...,x_N\}$. Let $A\ge 16$ be such that $\pr[X=x_i]\ge \frac{1}{A}$ for all $i\in[N]$. Let $J\subseteq\{x_1,...,x_N\}$ and $\cE$ be the event that $X\in J$. Suppose that $\pr[\cE]\ge 1-a$ for some $0\le a\le\frac{1}{2}$. Then
$$H(X)-H(X|\cE)\le 2a\log A.$$
\end{fact}

\subsection{Auxiliary graph}\label{section: auxiliary graph}
To build a path in a $k$-graph~$G$, the choice of each vertex is constrained by the preceding $k-1$ vertices. To facilitate many of our arguments, it is therefore useful to consider the interactions among the $(k-1)$-tuples of $V(G)$. To this end, we introduce an auxiliary graph on such tuples and study its properties in the following definitions and lemmas.

Let $G$ be an $(n,k,\delta)$-graph and $\bx$ be a perfect fractional matching of $G$. For any $S\subseteq V(G)$ with $|S|\in [k]_0$, we define the quantity \mbox{$\bx_S\coloneqq \sum_{S\subseteq e\in E(G)}\bx_e$}, which is the total weight assigned by $\bx$ to the edges containing $S$. Define \mbox{$\cS(G)\coloneqq {V(G)\choose {k-1}}$} to be the set of all $(k-1)$-subsets of $V(G)$. Let $\vec{\cS}(G)$ denote the set of all $(k-1)$-tuples in $V(G)^{k-1}$ with distinct entries. For each \mbox{$\vec S=(s_1,...,s_{k-1})\in\vec{\cS}(G)$}, we write $S=\{s_1,...,s_{k-1}\}$ for the underlying $(k-1)$-set in $\cS(G)$. Note that for all $S\in\cS(G)$, $\bx_S\coloneqq \sum_{v\in N_G(S)}\bx[S\cup\{v\}]$.

We define $\mathbb{S}[G]$ to be the weighted digraph with vertex set $\vec{\cS}(G)$, where for all $\vec S=(s_1,...,s_{k-1}), \vec T=(t_1,...,t_{k-1})\in \vec{\cS}(G)$, there is a directed edge $\vec S\vec T\in E(\mathbb{S}[G])$ if $(s_2,...s_{k-1})=(t_1,...,t_{k-2})$ and $S\cup\{t_{k-1}\}\in E(G)$. The edge weighting $\by:E(\mathbb{S}[G])\to[0,1]$ in $\mathbb{S}[G]$ induced by $\bx$ is defined by
$$\by[\vec S\vec T]\coloneqq\begin{cases}\frac{\bx[S\cup T]}{\bx_S}&\text{if }\bx_S\neq 0\\
\frac{1}{\deg_G(S)}&\text{if }\bx_S=0\end{cases}$$
for all $\vec S\vec T\in E(\mathbb{S}[G])$. As we focus on the scenario where $\bx$ is $b$-normal for some $b\ge 1$, we always have $\bx_S\neq 0$. Note that when an edge $\vec S\vec T$ exists in $\mathbb{S}[G]$, its weight is independent of the ordering of the entries in $\vec S$ and $\vec T$. Therefore, for simplicity, we often write $\by[S\to t_{k-1}]$ for $\by[\vec S\vec T]$. We further observe that for each $\vec S\in \vec{\cS}(G)$, we have \mbox{$\sum_{\vec T\in N^+_{\mathbb{S}[G]}(\vec S)}\by[\vec S\vec T]=\sum_{v\in N_G(S)}\by[S\to v]=1$}, where $N^+_{\mathbb{S}[G]}(\vec S)$ denotes the out-neighbourhood of $\vec S$ in $\mathbb{S}[G]$. So we may interpret $\by$ as a probability distribution on the outgoing edges of $\vec S\in \vec{\cS}(G)$.

For each $S\in \cS(G)$, we define the \emph{entropy of $S$ induced by $\bx$} by
$$h_{\bx}(S)\coloneqq \sum_{v\in N_G(S)}\by[S\to v]\log\frac{1}{\by[S\to v]}.$$

Note that there is a natural bijection between the directed walks in $\mathbb{S}[G]$ and the walks in~$G$, by tracing the same sequence of vertices in $G$. Let $W$ be a walk in~$G$, and $W'$ be the corresponding directed walk in $\mathbb{S}[G]$. If $W'$ is a walk from $\vec S$ to $\vec T$, then we also say that $W$ is a walk from $\vec S$ to $\vec T$.

We collect some useful properties in the following lemma.

\begin{lemma}
Suppose $n\ge k\ge 2$. Let $G$ be an $(n,k,\delta)$-graph and $\bx$ be a perfect fractional matching of $G$. Then the following holds.
\begin{enumerate}[label=\normalfont(\roman*), ref={\thelemma(\roman*)}]
\item $\sum_{e\in E(G)}\bx_e=n/k,$\label{lemma: S[G] property i}
\item $\sum_{S\in \cS(G)}\bx_S=n$, and\label{lemma: S[G] property ii}
\item $\sum_{v\in S\in \cS(G)}\bx_S=k-1$ for every vertex $v\in V(G)$.\label{lemma: S[G] property iii}
\end{enumerate}
If we further assume that $\bx$ is $b$-normal for some $b\ge1$, then
\begin{enumerate}[label=\normalfont(\roman*), ref={\thelemma(\roman*)}]\setcounter{enumi}{3}
\item $\sum_{S\in\cS(G)}\bx_S h_{\bx}(S)=kh(\bx)+\sum_{S\in\cS(G)}\bx_S\log\bx_S\ge kh(\bx)-n\log{n\choose k-1}+n\log n$,\label{lemma: S[G] property vi}
\item $\frac{\delta}{bn^{k-2}}\le\bx_S\le\frac{b}{n^{k-2}}$ for all $S\in\cS(G)$ and\label{lemma: S[G] property iv}
\item $\frac{1}{b^2n}\le\by[\vec S\vec T]\le\frac{b^2}{\delta n}$
for all $\vec S\vec T\in E(\mathbb{S}[G])$.\label{lemma: S[G] property v}
\end{enumerate}
\end{lemma}

\begin{proof}
We have
$$k\sum_{e\in E(G)}\bx_e=\sum_{e\in E(G)}\sum_{v\in e}\bx_e=\sum_{v\in V(G)}\sum_{v\in e\in E(G)}\bx_e=\sum_{v\in V(G)}1=n.$$
Dividing both sides by $k$ yields (i). Next, (ii) and (iii) follow from the following calculations.
\begin{align*}
\sum_{S\in \cS(G)}\bx_S&=\sum_{S\in \cS(G)}\sum_{v\in N_G(S)}\bx[S\cup\{v\}]=k\sum_{e\in E(G)}\bx_e=n,\text{ and}\\
\sum_{v\in S\in \cS(G)}\bx_S&=\sum_{v\in S\in \cS(G)}\sum_{u\in N_G(S)}\bx[S\cup\{u\}]=(k-1)\sum_{v\in e\in E(G)}\bx_e=k-1
\end{align*}
for every vertex $v\in V(G)$. Next, suppose $\bx$ is $b$-normal. To prove (iv), we calculate
\begin{align*}
\sum_{S\in\cS(G)}\bx_S h_{\bx}(S)&=\sum_{S\in\cS(G)}\bx_S\sum_{v\in N_G(S)}\by[ S\to v]\log\frac{1}{\by[S\to v]}\\
&=\sum_{S\in\cS(G)}\sum_{v\in N_G(S)}\bx[ S\cup\{v\}]\log\frac{\bx_S}{\bx[S\cup \{v\}]}\\
&=k\sum_{e\in E(G)}\bx_e\log\frac{1}{\bx_e}+\sum_{S\in\cS(G)}\sum_{v\in N_G(S)}\bx[ S\cup\{v\}]\log\bx_S\\
&=kh(\bx)+\sum_{S\in\cS(G)}\bx_S\log\bx_S\\
&\ge kh(\bx)-n\log{n\choose k-1}+n\log n,
\end{align*}
where the last step follows from Jensen's inequality applied to the concave function $t\mapsto \log t$, as shown below.
\begin{align*}
\sum_{S\in\mathcal{S}(G)}\bx_S\log \frac{1}{\bx_S}&\le \left(\sum_{S\in\mathcal{S}(G)}\bx_S\right)\log\frac{\sum_{S\in\mathcal{S}(G)}\bx_S\frac{1}{\bx_S}}{\sum_{S\in\mathcal{S}(G)}\bx_S}=n\log \frac{{n\choose k-1}}{n}.
\end{align*}
This proves (iv). For each $S\in \cS(G)$, we have
$$\frac{\delta}{bn^{k-2}}\le\delta_{k-1}(G)\frac{1}{bn^{k-1}}\le\sum_{v\in N_G(S)}\bx[S\cup\{v\}]\le n\frac{b}{n^{k-1}}=\frac{b}{n^{k-2}}.$$
This proves (v). For each $\vec S\vec T\in E(\mathbb{S}[G])$, we exploit the previous inequality and obtain
$$\frac{1}{b^2n}=\frac{1}{bn^{k-1}}\frac{n^{k-2}}{b}\le\by[\vec S\vec T]\le\frac{b}{n^{k-1}}\frac{bn^{k-2}}{\delta}= \frac{b^2}{\delta n},$$
which proves (vi).
\end{proof}

\section{Subgraph entropy}\label{section: Subgraph entropy}

We now study how removing a subset of vertices affects the entropy of an $(n,k,\delta)$-graph. Lemma~\ref{lemma: entropy is as expected after removing a well-behaved subset} shows that if a set interacts with the neighbourhood of each $S\in\cS(G)$ roughly in the way we would heuristically expect it from a randomly chosen set, then the entropy of the graph also decreases roughly as expected after its removal. The proofs in this section follow the same general strategy as those of Lemmas 4.5 and 4.6 in~\cite{JS:24} for digraphs.

\begin{lemma}\label{lemma: entropy is as expected after removing a well-behaved subset}
Suppose $1/n\ll\alpha\ll \hat\delta,1/k,1/b$ where $b\ge 4$. Let $G$ be an $(n,k,\delta)$-graph and $\bx$ be a $b$-normal perfect fractional matching of $G$. Let $M\subseteq V(G)$ be of size $m\coloneqq n^{1/3}$. Suppose the following holds for all $S\in\cS(G)$.
\begin{enumerate}[label=\normalfont(\roman*), ref={\thelemma(\roman*)}]
    \item $\sum_{v\in M\cap N_G(S)}\by[S\to v]=(1\pm n^{-\alpha})\frac{m}{n}$, and
    \item $\sum_{v\in M\cap N_G(S)}\by[S\to v]\log\frac{1}{\by[S\to v]}=\frac{m}{n}h_{\bx}(S)\pm n^{-\frac{2}{3}-\alpha}$.
\end{enumerate}
Then there is a $(1+n^{-\frac{2}{3}-\frac{\alpha}{2}})b$-normal perfect fractional matching $\bz$ of $G-M$ that satisfies
$$h(\bz)\ge\frac{n-m}{n}h(\bx)-\frac{k-1}{k}(n-m)\log \frac{n}{n-m}-n^{\frac{1}{3}-\frac{\alpha}{2}}.$$
\end{lemma}

\begin{proof}
Let $n'\coloneqq n-m$, $G'\coloneqq G-M$ and $b'\coloneqq (1+n^{-\frac{2}{3}-\frac{\alpha}{2}})b$. We first scale the weights on $E(G')$ to obtain an edge weighting $\bx'$ with total weight $n'/k$, which, by Lemma~\ref{lemma: S[G] property i}, is the total weight of any perfect fractional matching of $G'$. We then define a procedure to slightly perturb $\bx'$ to obtain a perfect fractional matching $\bz$. Lastly, we verify that $\bz$ is $b'$-normal and has the desired entropy. Let
$$\Lambda\coloneqq \frac{n'/k}{\sum_{e\in E(G')}\bx_e}$$
be the scaling factor. We define an edge weighting $\bx'$ of $G'$ by $\bx'_e\coloneqq \Lambda\bx_e>0$ for all $e\in E(G')$. Let $E(M,G)\coloneqq \{e\in E(G)|e\cap M\neq \emptyset\}$. Note that
$$\Lambda=\frac{n'/k}{\sum_{e\in E(G)}\bx_e-\sum_{e\in E(M,G)}\bx_e}=\frac{n'/k}{n/k-\sum_{e\in E(M,G)}\bx_e}=\frac{n-m}{n-km}\frac{1}{1+\frac{k(m-\sum_{e\in E(M,G)}\bx_e)}{n-km}}.$$
We study the sum $\sum_{e\in E(M,G)}\bx_e$ to obtain a bound for $\Lambda$. Note that
$$\sum_{S\in\cS(G')}\sum_{v\in M \cap N_G(S)}\bx[S\cup\{v\}]\le \sum_{e\in E(M,G)}\bx_e\le \sum_{S\in \cS(G)}\sum_{v\in M \cap N_G(S)}\bx[S\cup\{v\}].$$
Recall that $\by:E(\mathbb{S}[G])\to[0,1]$ is the edge weighting in $\mathbb{S}[G]$ given by
$\by[\vec S\vec T]\coloneqq\frac{\bx[S\cup T]}{\bx_S}$ for all $\vec S\vec T\in E(\mathbb{S}[G])$. Using Lemmas~\ref{lemma: S[G] property ii} and~\ref{lemma: S[G] property iv}, we have
\begin{align*}
\sum_{e\in E(M,G)}\bx_e&=\sum_{S\in \cS(G)}\sum_{v\in M \cap N_G(S)}\bx[S\cup\{v\}]\pm\sum_{S\in \cS(G)\colon S\cap M\neq \emptyset}\sum_{v\in M \cap N_G(S)}\bx[S\cup\{v\}]\\
&=\sum_{S\in \cS(G)}\bx_S \sum_{v\in M \cap N_G(S)}\by[S\to v]\pm\sum_{S\in \cS(G)\colon S\cap M\neq \emptyset}\bx_S \sum_{v\in M \cap N_G(S)}\by[S\to v]\\
&=n\cdot (1\pm n^{-\alpha})\frac{m}{n}\pm n^{k-2}m\cdot \frac{b}{n^{k-2}}\cdot(1+ n^{-\alpha})\frac{m}{n}\\
&=\left(m\pm \frac{m}{n^{\alpha}}\right)\pm \frac{2bm^2}{n}=m\pm n^{\frac{1}{3}-\frac{11\alpha}{12}}.
\end{align*}
Putting these together yields a precise estimate for $\Lambda$,
\begin{equation}\label{eq: Lambda bound}
\Lambda=\frac{n-m}{n-km}\frac{1}{1\pm\frac{kn^{\frac{1}{3}-\frac{11\alpha}{12}}}{n-km}}=\left(1+\frac{(k-1)m}{n-km}\right)\left(1\pm\frac{2kn^{\frac{1}{3}-\frac{11\alpha}{12}}}{n}\right)=1+\frac{(k-1)m}{n-km}\pm n^{-\frac{2}{3}-\frac{10\alpha}{12}}.
\end{equation}

Next, we slightly redistribute the weights of $\bx'$ to obtain a perfect fractional matching. We proceed by induction on the number of vertices $v\in V(G')$ with $\bx_v'\neq 1$. If no such vertices exist, then $\bx'$ is a perfect fractional matching of $G'$. Otherwise, there exists a pair of vertices $v, u\in V(G')$ such that $\bx'_v> 1> \bx'_u$. Observe that $\delta_{k-1}(G')\ge(\frac{1}{2}+\frac{\hat\delta}{2})n'$ and thus by monotonicity of minimum degrees (Lemma~\ref{lemma: monotonicity of minimum degrees}), the vertices $v$ and $u$ have at least
\begin{equation}\label{eq: gamma def}
N_{G'}(v)\cap N_{G'}(u)\ge \hat\delta{n'-1\choose k-1}\eqqcolon N
\end{equation}
common neighbours in $G'$. We pick a set $\Gamma$ of $(k-1)$-sets in $N_{G'}(v)\cap N_{G'}(u)$ with $|\Gamma|=N$. Let \mbox{$d\coloneqq\min\left\{\bx'_v-1, 1-\bx'_u\right\}$}. For all $S\in \Gamma$, we perform a ``weight shift" as follows. We decrease the weight of the edge $S\cup \{v\}$ by $d/N$ and increase the weight of the edge $S\cup \{u\}$ by $d/N$. As a result, the total weight at~$v$ is decreased by~$d$, the total weight at $u$ is increased by~$d$, and the total weight at all other vertices remains unchanged. By definition of $d$, at least one of $v$ and $u$ now has weight 1, as desired.

We repeat this process until the total weight at every vertex of $G'$ is 1 and denote the resulting weighting by $\bz$. We refer to the above transformation of $\bx'$ into $\bz$ as the \emph{normalization procedure}.

We now prove that $\bz$ is $b'$-normal, which in particular implies that $\bz>0$ and hence that $\bz$ is a perfect fractional matching of $G'$. The calculation below also shows that the normalization procedure is valid, in the sense that the edge weighting at each intermediate step is $b'$-normal, since the bounds in (\ref{eq: z bound}) hold throughout the procedure. Let $\Delta\coloneqq \max_{v\in V(G')}\left|\bx_v'-1\right|$. For every edge, the amount of weight redistributed due to each of its vertices is at most $\Delta/N$. Since each edge has $k$ vertices, we have
\begin{equation}\label{eq: z bound}
\bz_e=\bx_e'\pm k\frac{\Delta}{N}=\Lambda \bx_e\pm k\frac{\Delta}{N}.
\end{equation}

We proceed to find bounds for $\Delta$. Similar to our estimation of $\sum_{e\in E(M,G)}\bx_e$, using Lemmas~\ref{lemma: S[G] property iii} and~\ref{lemma: S[G] property iv}, for every $v\in V(G')$, we have
\begin{align*}
&\mathrel{}\sum_{v\in e\in E(M,G)}\bx_e\\
=&\mathrel{}\sum_{v\in S\in \cS(G)}\bx_S\sum_{u\in M \cap N_G(S)}\by[S\to u]\pm\sum_{S\in \cS(G)\colon v\in S, S\cap M\neq\emptyset}\bx_S \sum_{u\in M \cap N_G(S)}\by[S\to u]\\
=&\mathrel{}(k-1)\cdot (1\pm n^{-\alpha})\frac{m}{n}\pm n^{k-3}m\cdot\frac{b}{n^{k-2}}\cdot (1+ n^{-\alpha})\frac{m}{n}\\
=&\mathrel{}\left(\frac{(k-1)m}{n}\pm \frac{(k-1)m}{n^{1+\alpha}}\right)\pm \frac{2bm^2}{n^2}=\mathrel{}\frac{(k-1)m}{n}\pm n^{-\frac{2}{3}-\frac{11\alpha}{12}}.
\end{align*}
Together with~(\ref{eq: Lambda bound}), this yields
\begin{align*}
\sum_{v\in e\in E(G')}\bx_e'&=\Lambda\sum_{v\in e\in E(G')}\bx_e=\Lambda\left(\sum_{v\in e\in E(G)}\bx_e-\sum_{v\in e\in E(M,G)}\bx_e\right)\\
&=\left(1+\frac{(k-1)m}{n-km}\pm n^{-\frac{2}{3}-\frac{10\alpha}{12}}\right)\left(1-\frac{(k-1)m}{n}\pm n^{-\frac{2}{3}-\frac{11\alpha}{12}}\right)\\
&=1\pm n^{-\frac{2}{3}-\frac{9\alpha}{12}}.
\end{align*}
That is, $\Delta\le n^{-\frac{2}{3}-\frac{9\alpha}{12}}$. Recall the definition of $N$ in~(\ref{eq: gamma def}). Hence, we obtain
\begin{equation}\label{eq: weight change  bound}
k\frac{\Delta}{N}\le \frac{kn^{-\frac{2}{3}-\frac{9\alpha}{12}}}{\hat\delta{n-m-1\choose k-1}}\le n^{-\frac{2}{3}-(k-1)-\frac{8\alpha}{12}}.
\end{equation}
We further observe that
\begin{align*}
\left(\Lambda\pm k\frac{\Delta}{N}bn^{k-1}\right)\frac{(n')^{k-1}}{n^{k-1}}&=\left(\left(1+\frac{(k-1)m}{n-km}\pm n^{-\frac{2}{3}-\frac{10\alpha}{12}}\right)\pm bn^{-\frac{2}{3}-\frac{8\alpha}{12}}\right)\left(1-\frac{m}{n}\right)^{k-1}\\
&=\left(1+\frac{(k-1)m}{n-km}\pm 2bn^{-\frac{2}{3}-\frac{8\alpha}{12}}\right)\left(1-\frac{(k-1)m}{n}\pm \frac{k^2m^2}{n^2}\right)\\
&=1\pm n^{-\frac{2}{3}-\frac{7\alpha}{12}}.
\end{align*}
Using~(\ref{eq: z bound}) and the fact that $\bx$ is $b$-normal, for each edge $e\in E(G')$, we have
\begin{align*}\bz_e&\le \Lambda \bx_e+k\frac{\Delta}{N}\le\Lambda\frac{b}{n^{k-1}}+k\frac{\Delta}{N}b^2=\left(\Lambda+ k\frac{\Delta}{N}bn^{k-1}\right)\frac{(n')^{k-1}}{n^{k-1}}\frac{b}{(n')^{k-1}}\le \frac{\left(1+n^{-\frac{2}{3}-\frac{6\alpha}{12}}\right)b}{(n')^{k-1}},\text{ and}\\
\bz_e&\ge\Lambda \bx_e-k\frac{\Delta}{N}\ge\Lambda\frac{1}{bn^{k-1}}-k\frac{\Delta}{N}=\left(\Lambda-k\frac{\Delta}{N}bn^{k-1}\right)\frac{(n')^{k-1}}{n^{k-1}}\frac{1}{b(n')^{k-1}}\\
&\ge\left(1-n^{-\frac{2}{3}-\frac{7\alpha}{12}}\right)\frac{1}{b(n')^{k-1}}\ge\frac{1}{\left(1+n^{-\frac{2}{3}-\frac{6\alpha}{12}}\right)b(n')^{k-1}}.
\end{align*}
This shows that $\bz$ is a $b'$-normal perfect fractional matching of $G'$.

It remains to find a lower bound on the entropy $h(\bz)$. Consider two edges $e, f\in E(G')$ with weights $w_e$ and $w_f$ respectively. Suppose a weight of $0<d'\le \Delta/N$ is redistributed from $e$ to $f$. That is, the new weights of the edges $e$ and $f$ are respectively $w_e-d'$ and $w_f+d'$. Assume that $w_e, w_e-d', w_f, w_f+d'$ are all $b'$-normal. Then by (\ref{eq: weight change  bound}), $d'\le n^{-\frac{2}{3}-(k-1)-\frac{8\alpha}{12}}$.
Note that $w_e-d'\ge w_e/2$ and $w_f+d'\le 2w_f.$ So, we have
\begin{align*}
-\left(\log\frac{1}{w_e-d'}-\log\frac{1}{w_f+d'}\right)&\ge-\left(\log\frac{2}{w_e}-\log\frac{1}{2w_f}\right)=\log w_e-\log w_f-2\\
&\ge \log \frac{1}{b'n^{k-1}}-\log \frac{b'}{n^{k-1}}-2\ge -3\log b'.
\end{align*}
Thus, the updated entropy of $\{e,f\}$ is
\begin{align*}
&\mathrel{}(w_e-d')\log\frac{1}{w_e-d'}+(w_f+d')\log\frac{1}{w_f+d'}\\
\ge&\mathrel{}w_e\log\frac{1}{w_e-d'}+w_f\log\frac{1}{w_f+d'}-d'\left(\log\frac{1}{w_e-d'}-\log\frac{1}{w_f+d'}\right)\\
\ge&\mathrel{}w_e\log\frac{1}{w_e}+w_f\log\frac{1}{w_f}-w_f\log\left(1+\frac{d'}{w_f}\right)-3d'\log b'\\
\ge&\mathrel{}w_e\log\frac{1}{w_e}+w_f\log\frac{1}{w_f}-w_f\cdot\frac{2d'}{w_f}-3d' \log b'\\
\ge&\mathrel{}w_e\log\frac{1}{w_e}+w_f\log\frac{1}{w_f}-4\frac{\Delta}{N} \log b'.
\end{align*}

Note that the normalization procedure involves weight shifts between at most $nN$ pairs of edges. Moreover, the edge weighting at each intermediate step is $b'$-normal. So we have
\begin{equation}\label{eq: h(z) bound}
h(\bz)\ge h(\bx')-nN\left(4\frac{\Delta}{N}\log b'\right)=h(\bx')-4n\Delta\log b'\ge h(\bx')-n^{\frac{1}{3}-\frac{8\alpha}{12}}.
\end{equation}
We now study $h(\bx')$. We first consider the following sum and use Lemmas~\ref{lemma: S[G] property ii}, \ref{lemma: S[G] property vi} and~\ref{lemma: S[G] property iv}.
\begin{align*}
&\mathrel{}\sum_{e\in E(M,G)}\bx_e\log\frac{1}{\bx_e}\\
\le&\mathrel{} \sum_{S\in \cS(G)}\sum_{v\in M\cap N_G(S)}\bx[S\cup \{v\}]\log\frac{1}{\bx[S\cup\{v\}]}\\
=&\mathrel{}\sum_{S\in \cS(G)}\bx_S\sum_{v\in M\cap N_G(S)}\by[S\to v]\log\frac{1}{\by[S\to v]}+\sum_{S\in \cS(G)}\bx_S\log\frac{1}{\bx_S}\sum_{v\in M\cap N_G(S)}\by[S\to v]\\
\le&\mathrel{}\sum_{S\in \cS(G)}\bx_S\left(\frac{m}{n}h_{\bx}(S)+n^{-\frac{2}{3}-\alpha}\right)+\sum_{S\in \cS(G)}\bx_S\log\frac{1}{\bx_S}\left(\frac{m}{n}+n^{-\frac{2}{3}-\alpha}\right)\\
=&\mathrel{}\frac{m}{n}\left(kh(\bx)+\sum_{S\in\cS(G)}\bx_S\log\bx_S-\sum_{S\in \cS(G)}\bx_S\log\bx_S\right)+ \sum_{S\in \cS(G)}\bx_S n^{-\frac{2}{3}-\alpha}\left(1+\log\frac{1}{\bx_S}\right)\\
\le&\mathrel{}\frac{m}{n}kh(\bx)+ \sum_{S\in \cS(G)}\bx_S n^{-\frac{2}{3}-\alpha}\left(1+\log\left(\frac{bn^{k-2}}{\delta}\right)\right)\\
\le&\mathrel{}\frac{m}{n}kh(\bx)+ n^{\frac{1}{3}-\frac{11\alpha}{12}}.\addtocounter{equation}{1}\tag{\theequation}\label{eq: decrease in entropy}
\end{align*}
Next, we bound $\log \Lambda$. Recall~(\ref{eq: Lambda bound}). Observe that
\begin{align*}
\Lambda-\left(\frac{n}{n'}\right)^{k-1}&= \Lambda-\left(1+\frac{m}{n'}\right)^{k-1}\le1+\frac{(k-1)m}{n-km}+ n^{-\frac{2}{3}-\frac{10\alpha}{12}} -\left(1+\frac{(k-1)m}{n-m}\right)\\
&=\frac{(k-1)m}{n}\left(\frac{1}{1-\frac{km}{n}}-\frac{1}{1-\frac{m}{n}}\right)+n^{-\frac{2}{3}-\frac{10\alpha}{12}}\le n^{-\frac{2}{3}-\frac{9\alpha}{12}}.
\end{align*}
So, $\Lambda\le(\frac{n}{n'})^{k-1}(1+n^{-\frac{2}{3}-\frac{9\alpha}{12}}(\frac{n'}{n})^{k-1})$, and hence, we have
\begin{align*}
\log \Lambda&\le \log \left[\left(\frac{n}{n'}\right)^{k-1}\left(1+n^{-\frac{2}{3}-\frac{9\alpha}{12}}\left(\frac{n'}{n}\right)^{k-1}\right)\right]\\
&\le (k-1)\log \frac{n}{n'}+\log \left(1+n^{-\frac{2}{3}-\frac{9\alpha}{12}}\right)\le (k-1)\log \frac{n}{n'}+2n^{-\frac{2}{3}-\frac{9\alpha}{12}}.
\end{align*}
Recall that $h(\bx)=\Theta_b(n\log n)$ by Corollary~\ref{cor: pfm entropy bound}. Therefore, using~(\ref{eq: decrease in entropy}), we have
\begin{align*}
&\mathrel{}h(\bx')\\
=&\mathrel{}\sum_{e\in E(G')}\Lambda\bx_e\log\frac{1}{\Lambda\bx_e}\\
=&\mathrel{}\Lambda\sum_{e\in E(G')}\bx_e\log\frac{1}{\bx_e}-\Lambda\sum_{e\in E(G')}\bx_e\log\Lambda\\
=&\mathrel{}\Lambda\left(h(\bx)-\sum_{e\in E(M,G)}\bx_e\log\frac{1}{\bx_e}\right)-\frac{n'}{k}\log\Lambda\\
\ge&\mathrel{} \left(1+\frac{(k-1)m}{n-km}-n^{-\frac{2}{3}-\frac{10\alpha}{12}}\right)\left(h(\bx)-\frac{m}{n}kh(\bx)-n^{\frac{1}{3}-\frac{11\alpha}{12}}\right)-\frac{n'}{k}\left((k-1)\log \frac{n}{n'}+2n^{-\frac{2}{3}-\frac{9\alpha}{12}}\right)\\
\ge&\mathrel{} \left(1+\frac{(k-1)m}{n-km}\right)\left(\frac{n'}{n}h(\bx)-\frac{(k-1)m}{n}h(\bx)\right)-\frac{k-1}{k}n'\log \frac{n}{n'}-n^{\frac{1}{3}-\frac{8\alpha}{12}}\\
\ge&\mathrel{}\frac{n'}{n}h(\bx)-\frac{k-1}{k}n'\log \frac{n}{n'}-n^{\frac{1}{3}-\frac{7\alpha}{12}}.
\end{align*}
Combining with~(\ref{eq: h(z) bound}) yields the desired result.
\end{proof}

The following lemma shows that when we remove a small enough subset of vertices from $G$, the entropy of the graph only decreases slightly.

\begin{lemma}\label{lemma: entropy decreases slightly after removing a small subset}
Suppose $1/n\ll \eps \ll\hat\delta,1/k$. Let $G$ be an $(n,k,\delta)$-graph and $M\subseteq V(G)$ be a vertex subset such that $|M|\le\frac{n}{\log^2 n}$. Then, $h(G-M)\ge h(G)-\eps n$.
\end{lemma}

\begin{proof}
Let $G'\coloneqq G-M$, $m\coloneqq |M|$ and $n'\coloneqq n-m$. By Lemma~\ref{lemma: existence of b-normal pfm}, there exists $b$ with $1/n\ll 1/b\ll \eps$ and $b\ge 4$ such that there is a $b$-normal perfect fractional matching $\bx$ of $G$ satisfying $h(\bx)\ge h(G)-\frac{\eps}{3}n$. Similar to the proof of Lemma~\ref{lemma: entropy is as expected after removing a well-behaved subset}, we first scale the weights on $E(G')$ and then perform the normalization procedure to obtain a perfect fractional matching of $G'$.

Let $E(M,G)\coloneqq \{e\in E(G)|e\cap M\neq \emptyset\}$. Define the scaling factor
$$\Lambda\coloneqq \frac{n'/k}{\sum_{e\in E(G')}\bx_e}=\frac{n'/k}{\sum_{e\in E(G)}\bx_e-\sum_{e\in E(M,G)}\bx_e}=\frac{n'/k}{\frac{n}{k}-\sum_{e\in E(M,G)}\bx_e}=\frac{n-m}{n-k\sum_{e\in E(M,G)}\bx_e}.$$
Define a new edge weighting $\bx_e'\coloneqq \Lambda \bx_e>0$ for each $e\in E(G')$. Note that
$$k\sum_{e\in E(M,G)}\bx_e=\sum_{e\in E(M,G)}\sum_{v\in e}\bx_e\ge \sum_{v\in M}\sum_{v\in e\in E(G)}\bx_e=m.$$
So, we have
$$\frac{m}{k}\le \sum_{e\in E(M,G)}\bx_e\le \sum_{v\in M}\sum_{v\in e\in E(G)}\bx_e=m.$$
Thus, $\Lambda$ satisfies
$$1=\frac{n-m}{n-m}\le\Lambda\le \frac{n-m}{n-km}\le 1+\frac{2km}{n}.$$
We apply the same normalization procedure as in the proof of Lemma~\ref{lemma: entropy is as expected after removing a well-behaved subset} to redistribute the weights of $\bx'$ until the total weight at each vertex in $G'$ is 1. We call the resulting edge weighting $\bz$. We now prove that $\bz$ is $b'$-normal for $b'\coloneqq(1+\frac{b^5m}{n})b$. It then follows that, in particular,  $\bz>0$ and hence that $\bz$ is a perfect fractional matching of $G'$. Let
$$\Delta\coloneqq \max_{v\in V(G')}\left|\bx_v'-1\right|\quad\text{and}\quad N\coloneqq\hat\delta{n'-1\choose k-1}\ge\eps n^{k-1}.$$
Then, as in~(\ref{eq: z bound}), we have
$$\bz_e=\Lambda \bx_e\pm k\frac{\Delta}{N}.$$
Observe that for all $v\in V(G')$, we have $|\{e\in E(M,G):v\in e\}|\le m\cdot n^{k-2}$. Hence, as $\bx$ is $b$-normal, this yields
$$0\le \sum_{v\in e\in E(M,G)}\bx_e\le mn^{k-2}\cdot\frac{b}{n^{k-1}}=\frac{bm}{n}.$$
Thus,
$$\bx_v'=\sum_{v\in e \in E(G')}\Lambda\bx_e=\Lambda\left(\sum_{v\in e \in E(G)}\bx_e-\sum_{v\in e \in E(M,G)}\bx_e\right)=\left(1\pm \frac{2km}{n}\right)\left(1\pm\frac{bm}{n}\right)=1\pm \frac{2bm}{n}.$$
Therefore, we have
$$\Delta\le \frac{2bm}{n}\quad\text{and}\quad
\frac{k\Delta}{N}\le\frac{k\cdot\frac{2bm}{n}}{\eps n^{k-1}}\le \frac{b^2m}{n^k}.$$
Observe that
$$\left(\Lambda\pm k\frac{\Delta}{N}bn^{k-1}\right)\frac{(n')^{k-1}}{n^{k-1}}\le\left(\left(1\pm \frac{2km}{n}\right)\pm \frac{b^2m}{n^k}\cdot bn^{k-1}\right)\left(1-\frac{m}{n}\right)^{k-1}\le1\pm \frac{b^4m}{n}.$$
We are now ready to show that $\bz$ is $b'$-normal. For each edge $e\in E(G')$, we have
\begin{align*}
\bz_e&\le \Lambda \bx_e+k\frac{\Delta}{N}\le\Lambda\frac{b}{n^{k-1}}+k\frac{\Delta}{N}b^2=\left(\Lambda+ k\frac{\Delta}{N}bn^{k-1}\right)\frac{(n')^{k-1}}{n^{k-1}}\frac{b}{(n')^{k-1}}\le\frac{\left(1+\frac{b^4 m}{n}\right)b}{(n')^{k-1}},\text{ and}\\
\bz_e&\ge\Lambda \bx_e-k\frac{\Delta}{N}\ge\Lambda\frac{1}{bn^{k-1}}-k\frac{\Delta}{N}=\left(\Lambda-k\frac{\Delta}{N}bn^{k-1}\right)\frac{(n')^{k-1}}{n^{k-1}}\frac{1}{b(n')^{k-1}}\\
&\ge\left(1-\frac{b^4 m}{n}\right)\frac{1}{b(n')^{k-1}}\ge\frac{1}{\left(1+\frac{b^5 m}{n}\right)b(n')^{k-1}}.
\end{align*}
Next, we find a lower bound for $h(\bz)$. We first consider
$$\sum_{e\in E(M,G)}\bx_e\log \frac{1}{\bx_e}\le mn^{k-1}\cdot\frac{b}{n^{k-1}}\log (bn^{k-1})\le \frac{\eps n}{3}$$
where the last step follows from the assumption that $m\le\frac{n}{\log^2 n}$. By the very same argument as~(\ref{eq: h(z) bound}), we have
\begin{align*}
h(\bz)&\ge h(\bx')-4n\Delta\log b'\ge \sum_{e\in E(G')}\bx_e'\log \frac{1}{\bx_e'}-4n\cdot \frac{2bm}{n}\cdot\log\left(1+\frac{b^5m}{n}\right)b\\
&\ge \sum_{e\in E(G')}\Lambda\bx_e\log \frac{1}{\Lambda\bx_e}-\frac{\eps n}{3}\ge \sum_{e\in E(G')}\bx_e\log \frac{1}{\bx_e}-\frac{\eps n}{3}\\
&\ge\sum_{e\in E(G)}\bx_e\log \frac{1}{\bx_e}-\sum_{e\in E(M,G)}\bx_e\log \frac{1}{\bx_e}-\frac{\eps n}{3}\ge h(\bx)-\frac{2\eps n}{3}\ge h(G)-\eps n.
\end{align*}
where the second line follows from that the function $t\mapsto t\log(1/t)$ is increasing when $t\le e^{-1}$ and we may choose $n$ large enough so that $\bx_e,\Lambda\bx_e\le e^{-1}$.
\end{proof}

\section{Probabilistic properties of random walks}\label{section: Probabilistic properties of random walks}

In this section, we study random walks in an $(n,k,\delta)$-graph and in its auxiliary graph defined in Section~\ref{section: auxiliary graph}. Results of this type can also be found in other papers on this topic; for example, for digraphs~{\cite[Section 5]{JS:24}}.

Let $G$ be an $(n,k,\delta)$-graph and $\bx$ be a perfect fractional matching of $G$. Let $\mathbb{S}[G]$ be as defined in Section 2.5 and $\by$ be the corresponding edge weighting induced by $\bx$. A \emph{random walk} on $\mathbb{S}[G]$ induced by $\bx$ is the Markov chain $(\vec Z_t)_{t\in\bN_0}$ with state space $\vec{\cS}(G)$ and a transition matrix $P=(p_{\vec S\vec T})_{\vec S, \vec T\in\vec{\cS}(G)}$, defined by
$$p_{\vec S\vec T}=\begin{dcases}
\by[\vec S\vec T]&\quad \text{if }\vec S\vec T\in E(\mathbb{S}[G])\\
0&\quad\text{otherwise}.
\end{dcases}$$
A \emph{random walk} on $G$ induced by $\bx$ is a sequence $(R_t)_{t\in \bN_0}$ of random variables with sample space~$V(G)$ such that if $\vec Z_t\coloneqq (R_t,...,R_{t+k-2})$ for all $t\in \bN_0$, then $(\vec Z_t)_{t\in \bN_0}$ is a random walk on $\mathbb{S}[G]$ induced by $\bx$. We note that every random walk on $\mathbb{S}[G]$ yields a random walk on $G$.

\subsection{Probabilistic tools}
We need the following probabilistic results for the remainder of this paper. Lemma~\ref{lemma: rapid mixing} is the rapid mixing property, which provides a lower bound on the speed of convergence of a Markov chain to the stationary distribution. A sequence $Y_0,...Y_{\ell}$ of randoms variables is a \emph{martingale} if $Y_0$ is a fixed real number and $\ex[Y_i|Y_0,...,Y_{i-1}]=Y_{i-1}$ for all $i\in[\ell]$. Lemma~\ref{lemma: azuma} is Azuma's inequality, see for example~\cite{azuma:67}, which yields concentration estimates of a martingale with bounded differences. Lemma~\ref{lemma: chernoff} is a Chernoff-type bound that is used at the end to select a suitable subgraph for the absorption argument, see for example~\cite{JLR:00}.

\begin{lemma}[{\cite[Lemma~3.2]{JK:21}}]\label{lemma: rapid mixing}
Let $(Z_t)_{t\in \bN_0}$ be a Markov chain with state space $\{s_1,...,s_n\}$ and transition matrix $P=(p_{ij})_{i,j\in[n]}$. Suppose there is a probability distribution $\sigma=(\sigma_i)_{i\in [n]}$ such that $\sigma P=\sigma$. Let
$$c_1\coloneqq \inf_{i,j,k\in[n]}\frac{p_{ij}}{\sigma_k}\text{ and } c_2\coloneqq \sup_{i,j,k\in[n]}\frac{p_{ij}}{\sigma_k}.$$
If $c_1>0$, then for $t\ge 2+2c_1^{-1}\log c_2$, we have
$$\pr[Z_t=s_i]=\left(1\pm\left(1-\frac{c_1}{2}\right)^t\right)\sigma_i.$$
\end{lemma}

\begin{lemma}\label{lemma: azuma}
Let $Y_0,...Y_{\ell}$ be a martingale such that $|Y_i-Y_{i-1}|\le c$ for all $i\in [\ell]$. Then for all $t>0$, we have
$$\pr[|Y_{\ell}-Y_0|\ge t]\le 2\exp\left(-\frac{t^2}{2\ell c^2}\right).$$
\end{lemma}

\begin{lemma}\label{lemma: chernoff}
Let $t>0$ and $X$ be a random variable with a binomial or hypergeometric distribution. Suppose $\ex[X]>0$. Then
$$\pr[X\le \ex[X]-t]\le \exp\left(-\frac{t^2}{2\ex[X]}\right).$$
\end{lemma}

\subsection{Convergence of random walks}
We also need the following result showing that in an $(n,k,\delta)$-graph $G$, any $\vec T\in\vec{\cS}(G)$ can be reached from any $\vec S\in\vec{\cS}(G)$ in a fixed number of steps.

\begin{lemma}[{\cite[Lemma~5.3]{GGJKO:21}}]\label{lemma: well-connectedness}
Suppose $1/n\ll \zeta\ll 1/L\ll \hat\delta,1/k$. Let $G$ be an $(n,k,\delta)$-graph. For any $\vec S, \vec T\in\vec{\cS}(G)$, there are at least $\zeta n^{L-(k-1)}$ walks of length $L$ from $\vec S$ to $\vec T$ in $G$.
\end{lemma}

We now apply the rapid mixing property of Markov chains to the random walks in $\mathbb{S}[G]$.

\begin{cor}\label{cor: rapid mixing in S[G]}
Suppose $1/n\ll\alpha\ll\hat\delta,1/k,1/b$. Let $G$ be an $(n,k,\delta)$-graph and $\bx$ be a $b$-normal perfect fractional matching of $G$. Let $(\vec Z_t)_{t\in\bN_0}$ be a random walk on $\mathbb{S}[G]$ induced by $\bx$. Then, for any $(k-1)$-tuple $\vec S\in \vec{\cS}(G)$ and $t\ge\alpha^{-2}$, we have
$$\pr[\vec Z_t=\vec S]=\left(1\pm e^{-\alpha t}\right)\frac{\bx_S}{(k-1)!n}.$$
\end{cor}

\begin{proof}
Using the same notation as the beginning of Section~\ref{section: Probabilistic properties of random walks}, $(\vec Z_t)_{t\in\bN_0}$ is a Markov chain with state space $\vec{\cS}(G)$ and transition matrix $P=(p_{\vec S\vec T})_{\vec S, \vec T\in\vec{\cS}(G)}$. By Lemma~\ref{lemma: S[G] property v}, we have the following bounds on the nonzero entries $p_{\vec S\vec T}$ of $P$,
$$\frac{1}{b^2n}\le p_{\vec S\vec T}\le\frac{b^2}{
\delta n}.$$

By Lemma~\ref{lemma: well-connectedness}, the hypergraph $G$, and hence $\mathbb{S}[G]$, has at least $\zeta n^{L-(k-1)}$ walks of length $L$ from $\vec S$ to $\vec T$ for any $\vec S, \vec T\in \vec{\cS}(G)$, where \mbox{$1/n\ll\zeta\ll 1/L\ll\hat\delta, 1/k$}. This allows us to define a new Markov chain $(\vec Z_t')_{t\in\bN_0}$ on $\vec{\cS}(G)$ with a strictly positive transition matrix $Q\coloneqq P^L$. We have the following bounds on the entries of $Q=(q_{\vec S \vec T})_{\vec S, \vec T\in\vec{\cS}(G)}$,
\begin{equation}\label{eq: Q bounds}
\frac{\zeta}{b^{2L}n^{k-1}}=\zeta n^{L-(k-1)}\left(\frac{1}{b^2n}\right)^{L}\le q_{\vec S\vec T}\le n^{L-(k-1)}\left(\frac{b^2}{\delta n}\right)^{L}=\frac{b^{2L}}{\delta^L n^{k-1}}.
\end{equation}
We define $\sigma\coloneqq (\sigma_{\vec S})_{\vec S\in\vec{\cS}(G)}$ by
$$\sigma_{\vec S}\coloneqq\frac{\bx_S}{(k-1)!n}$$
for each $\vec S\in\vec{\cS}(G)$. We claim that $\sigma$ is the stationary distribution of the Markov chain $(Z_t')_{t\in\bN_0}$.

Indeed, using Lemma~\ref{lemma: S[G] property ii}, it is straightforward to verify that $\sigma$ is a probability distribution. It remains to show that $\sigma$ is a left eigenvector of $P$, and hence a left eigenvector of $Q$. For all $\vec T\in\vec{\cS}(G)$, we have
\begin{align*}
(\sigma P)_{\vec T}&=\sum_{\vec S\in\vec{\cS}(G)}\sigma_{\vec S}p_{\vec S\vec T}=\sum_{\vec S\in\vec{\cS}(G)}\sigma_{\vec S}\by[\vec S\vec T]\1[\vec S\vec T\in E(\mathbb{S}[G])]\\
&=\sum_{\vec S\in\vec{\cS}(G)}\frac{\bx_S}{(k-1)!n}\cdot\frac{\bx[S\cup \{t_{k-1}\}]}{\bx_S}\cdot\1[\vec S\vec T\in E(\mathbb{S}[G])]\\
&=\frac{1}{(k-1)!n}\sum_{\vec S\in\vec{\cS}(G)}\bx[\{s_1\}\cup T]\1 [\vec S\vec T\in E(\mathbb{S}[G])]\\
&=\frac{1}{(k-1)!n}\sum_{v\in N_G(T)}\bx[\{v\}\cup T]=\frac{\bx_T}{(k-1)!n}=\sigma_{\vec T},
\end{align*}
which proves our claim.

By Lemma~\ref{lemma: S[G] property iv}, for each $\vec S\in\vec{\cS}(G)$, we have
$$\frac{\delta}{b(k-1)!n^{k-1}}\le \sigma_{\vec S}\le \frac{b}{(k-1)!n^{k-1}}.$$
Together with~(\ref{eq: Q bounds}), this allows us to bound $c_1$ and $c_2$ as defined in Lemma~\ref{lemma: rapid mixing}. We have
\begin{align*}
c_1&\coloneqq \inf_{\vec S, \vec T, \vec U\in\vec{\cS}(G)}\frac{q_{\vec S\vec T}}{\sigma_{\vec U}}\ge\frac{\zeta}{b^{2L}n^{k-1}}\cdot\frac{(k-1)!n^{k-1}}{b}=\frac{\zeta(k-1)!}{b^{2L+1}}>0,\text{ and}\\
c_2&\coloneqq \sup_{\vec S, \vec T, \vec U\in\vec{\cS}(G)}\frac{q_{\vec S\vec T}}{\sigma_{\vec U}}\le\frac{b^{2L}}{\delta^{L}n^{k-1}}\cdot\frac{b(k-1)!n^{k-1}}{\delta}=\frac{b^{2L+1}(k-1)!}{\delta^{L+1}}.
\end{align*}
Let
$$\alpha\coloneqq\min\left\{\frac{\zeta(k-1)!}{2L b^{2L+1}},\frac{\delta^{L+1}}{b^{2L+1}(k-1)!}\right\}$$
so that $c_1\ge 2L\alpha$ and $c_2\le\alpha^{-1}$.
Using the inequality $(1-x)^t\le e^{-xt}$, Lemma~\ref{lemma: rapid mixing} yields
$$\pr[\vec Z_t=\vec S]=\pr[\vec Z_{t/L}'=\vec S]=\left(1\pm\left(1-\frac{c_1}{2}\right)^{\frac{t}{L}}\right)\sigma_{\vec S}=\left(1\pm\exp\left(-\frac{c_1 t}{2L}\right)\right)\sigma_{\vec S}=\left(1\pm e^{-\alpha t}\right)\sigma_{\vec S}$$
for any $\vec S\in\vec{\cS}(G)$ and any $t/L\ge 2+2(2L\alpha)^{-1}\log(\alpha^{-1})$ such that $L|t$. Similarly, we conclude that under the same conditions on $t$, we have
$$\pr[\vec Z_{i+t}=\vec S|\vec Z_i=\vec T]=\left(1\pm e^{-\alpha t}\right)\sigma_{\vec S}$$
for any $\vec S, \vec T\in\vec{\cS}(G)$. Therefore, we obtain
$$\pr[\vec Z_t=\vec S]=\sum_{\vec T\in\vec{\cS}(G)}\pr[\vec Z_{t\bmod L}=\vec T]\pr[\vec Z_t=\vec S|\vec Z_{t\bmod L}=\vec T]=\left(1\pm e^{-\alpha t}\right)\sigma_{\vec S}$$
for any $\vec S\in\vec{\cS}(G)$ and any $t\ge L(2+2(2L\alpha)^{-1}\log(\alpha^{-1}))+L-1=3L-1+\alpha^{-1}\log(\alpha^{-1})$. In particular, the bound holds for all $t\ge\alpha^{-2}$.
\end{proof}

We now show that a random walk on $G$ satisfies a similar rapid mixing property.

\begin{cor}\label{cor: rapid mixing in G}
Suppose $1/n\ll\alpha\ll\hat\delta,1/k,1/b$. Let $G$ be an $(n,k,\delta)$-graph and $\bx$ be a $b$-normal perfect fractional matching of $G$. Let $(R_t)_{t\in\bN_0}$ be a random walk on $G$ induced by $\bx$. Then, for any vertex $v\in V(G)$ and $t\ge \alpha^{-2}$, we have
$$\pr[R_t=v]=\left(1\pm e^{-\alpha t}\right)\frac{1}{n}.$$
\end{cor}

\begin{proof}
Let $(\vec Z_t)_{t\in \bN_0}$ be the corresponding random walk on $\mathbb{S}[G]$ induced by $\bx$ and choose $\alpha$ as in Lemma~\ref{cor: rapid mixing in S[G]}. Then for all $t\ge \alpha^{-2}$ and $v\in V(G)$, we have
\begin{align*}
\pr[R_t=v] &=\pr[\vec Z_t\in\{\vec S\in\vec{\cS}(G)|s_1=v\}]=\sum_{\vec S\in\vec{\cS}(G)\colon s_1=v}\pr[\vec Z_t=\vec S]\\
&=(k-2)!\sum_{v\in S\in\cS(G)}(1\pm e^{-\alpha t})\frac{\bx_S}{(k-1)!n}=(1\pm e^{-\alpha t})\frac{1}{n}
\end{align*}
where the last step follows from Lemma~\ref{lemma: S[G] property iii}.
\end{proof}

\subsection{Random walks and vertex subsets}
The following lemma shows that with high probability, a random walk in $G$ behaves as expected with regard to a subset $U\subseteq V(G)$ due to the rapid mixing property (Corollary~\ref{cor: rapid mixing in G}).

\begin{lemma}\label{lemma: behaviour wrt subsets}
Suppose $1/n\ll\alpha\ll\hat\delta, 1/k, 1/b$. Let $G$ be an $(n,k,\delta)$-graph and $\bx$ be a $b$-normal perfect fractional matching of $G$. Let $c\in \bR$ and let $\bw:V(G)\to \bR_{\ge0}$ be a vertex weighting with $\bw_v\le c$ for all $v\in V(G)$. Let $U\subseteq V(G)$. Let $R=(R_1,...,R_m)$ be a random walk on $G$ induced by $\bx$, where $m\coloneqq n^{1/3}$. Then
$$\pr\left[\left|\sum_{v\in R\cap U}\bw_{v}-\frac{m}{n}\sum_{v\in U}\bw_v\right|\ge cn^{\frac{1}{3}-\alpha}\right]\le \frac{1}{n^k}.$$
\end{lemma}

\begin{proof}
Choose $\alpha$ such that Corollary~\ref{cor: rapid mixing in G} holds. We partition $R$ into $\ell$ consecutive disjoint random walks \mbox{$Q_1=(Q_{1,1},... Q_{1,m_1})$}, $...$, \mbox{$Q_{\ell}=(Q_{\ell,1},...,Q_{\ell,m_{\ell}})$} such that \mbox{$m^{1-\alpha}\le|Q_i|\le 2m^{1-\alpha}$} for all \mbox{$i\in [\ell]$}. Then, we have $m^{\alpha}/2\le\ell\le m^{\alpha}$. For $i\in[\ell]$, let \mbox{$Q_i'\coloneqq(Q_{i,\ln m},...,Q_{i,m_i})$}; that is, we delete the first $\ln m$ vertices from $Q_i$ to obtain $Q_i'$. Since only a few vertices are removed in this process, their contribution is negligible in the final estimate. Then Corollary~\ref{cor: rapid mixing in G} applies to every step in $Q_i'$, since $\ln m>\alpha^{-2}$. We write \mbox{$Q_i-Q_i'=(Q_{i,1},...,Q_{i,\ln m-1})$}, then \mbox{$|Q_i-Q_i'|\le \ln m$}. Note that
\begin{equation}\label{eq: Q_i' sum bound}
m-m^{\alpha}\ln m\le|R|-\sum_{i=1}^{\ell}|Q_i-Q_i'|=\sum_{i=1}^{\ell}|Q_i'|\le |R|=m.
\end{equation}
For each $i\in [\ell]$, let $J_i\coloneqq\{j\in[m]|R_j\in Q_i-Q_i'\}$, $J_i'\coloneqq\{j\in[m]|R_j\in Q_i'\}$, and
$$W_i\coloneqq\sum_{j\in J_i'}\bw_{R_j}\1_{R_j\in U}\le c\sum_{j\in J_i'}\1_{R_j\in U}\le c|Q_i'|.$$

Consider the exposure martingale $Y_i\coloneqq \ex[\sum_{i'=1}^{\ell} W_{i'}|W_1,...,W_i]$, for $i\in[\ell]_0$. Our plan is to use Azuma's inequality to show that $Y_0$ and $Y_{\ell}$ are close to each other. Moreover, they are each close to one of the two sums in the statement of this lemma. Then we use triangle inequality to obtain the desired result.

Let $i\in[\ell]$ and $w_1,...,w_i,w_i'\in\bN_0$ be such that $w_{i'}\le c|Q_{i'}'|$ for each $i'\in[i]$ and $w_i'\le c|Q_i'|$. Let $\cE$ be the event that $W_1=w_1,...,W_i=w_i$ occurs, and let $\cE'$ be the event that $W_1=w_1,...,W_{i-1}=w_{i-1},W_i=w_i'$ occurs. Fix $i<i'\le\ell$. By Corollary~\ref{cor: rapid mixing in G}, for all $j\in J_{i'}'$ and $v\in V(G)$, we have
$$\pr[R_j=v|\cE]=(1\pm e^{-\alpha\ln m})\frac{1}{n}=(1\pm m^{-\alpha})\frac{1}{n}.$$
Using linearity of expectation, this implies
\begin{equation}\label{eq: expectation of W_i}
\ex[W_{i'}|\cE]=\sum_{j\in J_{i'}'}\sum_{v\in U}\bw_{R_j}\ex[\1_{R_j=v}|\cE]=\sum_{j\in J_{i'}'}\sum_{v\in U}\bw_v\pr[R_j=v|\cE]=|Q_{i'}'|\cdot(1\pm m^{-\alpha})\frac{1}{n}\sum_{v\in U}\bw_v.
\end{equation}
The same result holds when $\cE$ is replaced by $\cE'$. Thus, we have
\begin{align*}
&\mathrel{}\left|\ex\left[\left.\sum_{i'=1}^{\ell} W_{i'}\right|\cE\right]-\ex\left[\left.\sum_{i'=1}^{\ell} W_{i'}\right|\cE'\right]\right|\\
\le&\mathrel{}|w_i-w_i'|+\sum_{i'=i+1}^{\ell}|\ex[W_{i'}|\cE]-\ex[W_{i'}|\cE']|\\
\le&\mathrel{} c|Q_i'|+\sum_{i'=i+1}^{\ell}\left||Q_{i'}'|(1\pm m^{-\alpha})\frac{1}{n}\sum_{v\in U}\bw_v-|Q_{i'}'|(1\pm m^{-\alpha})\frac{1}{n}\sum_{v\in U}\bw_v\right|\\
\le&\mathrel{} c\cdot 2m^{1-\alpha}+2m^{-\alpha}\frac{1}{n}\sum_{i'=i+1}^{\ell}|Q_{i'}'|\sum_{v\in U}\bw_v\le 4cm^{1-\alpha}.
\end{align*}
Hence, $|Y_i-Y_{i-1}|\le 4cm^{1-\alpha}$ for all $i\in [\ell]$. Using Azuma's inequality and $\ell\le m^{\alpha}$, we have
\begin{equation}\label{eq: Y bound}
\pr[|Y_{\ell}-Y_0|\ge cm^{1-\frac{\alpha}{3}}]\le 2\exp\left(-\frac{\left(cm^{1-\frac{\alpha}{3}}\right)^2}{2\ell \left(4cm^{1-\alpha}\right)^2}\right)\le 2\exp\left(-\frac{m^{\alpha/3}}{32}\right)\le\frac{1}{n^k}.
\end{equation}
Similar to~(\ref{eq: expectation of W_i}), we obtain
$$Y_0=\ex\left[\sum_{i=1}^{\ell} W_i\right]=\sum_{i=1}^{\ell}\ex [W_i]=\sum_{i=1}^{\ell}|Q_i'|(1\pm m^{-\alpha})\frac{1}{n}\sum_{v\in U}\bw_v.$$
Using~(\ref{eq: Q_i' sum bound}), we have
\begin{align*}
Y_0&=(m\pm m^{\alpha}\ln m)(1\pm m^{-\alpha})\frac{1}{n}\sum_{v\in U}\bw_v\\
&=\frac{m}{n}\sum_{v\in U}\bw_v\pm(m^{\alpha}\ln m+m^{1-\alpha}+\ln m)\frac{1}{n}\sum_{v\in U}\bw_v\\
&=\frac{m}{n}\sum_{v\in U}\bw_v\pm 2cm^{1-\alpha}.\addtocounter{equation}{1}\tag{\theequation}\label{eq: Y_0 bound}
\end{align*}

Moreover, by construction, we have
$$Y_{\ell}\le \sum_{j=1}^m\bw_{R_j}\1_{R_j\in U}=\sum_{i=1}^{\ell}\sum_{j\in J_i'} \bw_{R_j}\1_{R_j\in U}+\sum_{i=1}^{\ell}\sum_{j\in J_i} \bw_{R_j}\1_{R_j\in U}\le\sum_{i=1}^{\ell} W_i+cm^{\alpha}\ln m\le Y_{\ell}+cm^{1-\alpha}.$$
Combining this with~(\ref{eq: Y_0 bound}), by triangle inequality, we have
\begin{align*}
\left|\sum_{j=1}^m\bw_{R_j}\1_{R_j\in U}-\frac{m}{n}\sum_{v\in U}\bw_v\right|&\le\left|\sum_{j=1}^m\bw_{R_j}\1_{R_j\in U}-Y_{\ell}\right|+|Y_{\ell}-Y_0|+\left|Y_0-\frac{m}{n}\sum_{v\in U}\bw_v\right|\\
&\le |Y_{\ell}-Y_0|+3cm^{1-\frac{\alpha}{3}}.
\end{align*}
So using~(\ref{eq: Y bound}), we conclude
\begin{align*}
\pr\left[\left|\sum_{v\in R\cap U}\bw_v-\frac{m}{n}\sum_{v\in U}\bw_v\right|\ge cn^{\frac{1}{3}-\frac{\alpha}{10}}\right]&\le\pr\left[\left|\sum_{j=1}^m\bw_{R_j}\1_{R_j\in U}-\frac{m}{n}\sum_{v\in U}\bw_v\right|\ge 4cm^{1-\frac{\alpha}{3}}\right]\\
&\le\pr[|Y_{\ell}-Y_0|\ge cm^{1-\frac{\alpha}{3}}]\le\frac{1}{n^k}.
\end{align*}
We obtain the desired result by redefining $\alpha$ to be $\alpha/10$.
\end{proof}

In particular, we apply the above lemma to obtain the following bounds on the number of vertices visited in a fixed set $U$ and how the transition probabilities and the entropy scale with respect to the neighbourhood of some $S\in \cS(G)$.

\begin{cor}\label{cor: a random walk is likely (U,a,a')-expected}
Suppose $1/n\ll\alpha\ll\hat\delta, 1/k, 1/b$. Let $G$ be an $(n,k,\delta)$-graph and $\bx$ be a $b$-normal perfect fractional matching of $G$. Let $R=(R_1,...,R_m)$ be a random walk on $G$ induced by $\bx$, where $m\coloneqq n^{1/3}$. Then the following holds.
\begin{enumerate}[label=\normalfont(\roman*), ref={\thelemma(\roman*)}]
    \item For every $U\subseteq V(G)$,
        $$\pr\left[\left||R\cap U|-\frac{m}{n}|U|\right|\ge n^{\frac{1}{3}-\alpha}\right]\le \frac{1}{n^k}.$$
    \item For every $S\in \cS(G)$,
        $$\pr\left[\left|\sum_{v\in R\cap N_G(S)}\by[S\to v]-\frac{m}{n}\right|\ge n^{-\frac{2}{3}-\alpha}\right]\le \frac{1}{n^k}.$$
    \item For every $S\in\cS(G)$,
        $$\pr\left[\left|\sum_{v\in R\cap N_G(S)}\by[S\to v]\log\frac{1}{\by[S\to v]}-\frac{m}{n}h_{\bx}(S)\right|\ge n^{-\frac{2}{3}-\alpha}\right]\le \frac{1}{n^k}.$$
\end{enumerate}
\end{cor}

\begin{proof}
Let $\alpha$ be such that Lemma~\ref{lemma: behaviour wrt subsets} holds with $2\alpha$ playing the role of $\alpha$. Using the same notation as in Lemma~\ref{lemma: behaviour wrt subsets}, (i) follows by choosing $\bw\equiv1$ and $c=1$.

Fix $S\in \cS(G)$. Using Lemma~\ref{lemma: S[G] property v}, we obtain (ii) by choosing $\bw_v=\by[S\to v]$ for all $v\in V(G)$, $c=\frac{b^2}{\delta n}$ and $U=N_G(S)$. Lastly, (iii) follows by choosing $\bw_v=\by[S\to v]\log\frac{1}{\by[S\to v]}$ for all $v\in V(G)$, $c=\frac{b^2}{\delta n}\log(b^2n)$ and $U=N_G(S)$.
\end{proof}

We now define some desirable properties of a random walk.

Let $G$ be an $(n,k,\delta)$-graph and $\bx$ be a perfect fractional matching of $G$. Let $0<\alpha<1/3$. We say that a set $M\subseteq V(G)$ is \emph{$(G,\alpha)$-well-behaved} if for each $S\in\cS(G)$, the following holds.

\begin{enumerate}[label=\normalfont(\roman*), ref={\thelemma(\roman*)}]
\item $\left||M\cap N_G(S)|-\frac{|M|}{n}|N_G(S)|\right|<n^{\frac{1}{3}-\alpha}$.
\item $\left|\sum_{v\in M\cap N_G(S)} \by[S\to v]-\frac{|M|}{n}\right|<n^{-\frac{2}{3}-\alpha}$.
\item $\left|\sum_{v\in M\cap N_G(S)} \by[S\to v]\log\frac{1}{\by[S\to v]}-\frac{|M|}{n}h_\bx(S)\right|<n^{-\frac{2}{3}-\alpha}$.
\end{enumerate}

Let $Z=(\vec Z_0,...,\vec Z_m)$ be a random walk on $\mathbb{S}[G]$ induced by $\bx$. We say that $Z$ is \emph{self-avoiding} if $Z_i\setminus Z_{i-1}\not\subseteq\cup_{j=0}^{i-1}Z_j$ for all $i\in [m]$. That is, the corresponding walk in $G$ is a path. Let $R=(R_1,...,R_m)$ be the random walk on $G$ corresponding to the first $m$ vertices in $Z$. We say that $Z$ is \emph{$(\mathbb{S}[G], \alpha)$-well-behaved} if $Z$ is self-avoiding and $R$ is $(G, \alpha)$-well-behaved. In this case, we also say that $R$ is a \emph{$(G, \alpha)$-well-behaved path}.

The following lemma shows that with high probability, a random walk is self-avoiding; to this end, a simple union bound suffices, compare \cite{CK:09:lb} and \cite{GGJKO:21}. We omit the short and obvious proof.

\begin{lemma}\label{lemma: a random walk is likely self-avoiding}
Let $G$ be an $(n,k,\delta)$-graph and $\bx$ be a perfect fractional matching of $G$. Let $m\ge 0$ and $Z=(\vec Z_0,...,\vec Z_m)$ be a random walk on $\mathbb{S}[G]$ induced by $\bx$. Suppose every edge $e\in E(G)$ satisfies $\bx_e\le b/n^{k-1}$ for some $b\ge 1$. Then
$$\pr[Z\text{ is self-avoiding}]\ge 1-\frac{2m^2b^2}{n}.$$
\end{lemma}

All together, we show that with high probability, a random walk is $(\mathbb{S}[G],\alpha)$-well-behaved.

\begin{cor}\label{cor: a random walk is likely good}
Suppose $1/n\ll\alpha\ll\hat\delta, 1/k, 1/b$. Let $G$ be an $(n,k,\delta)$-graph and $\bx$ be a $b$-normal perfect fractional matching of $G$. Let $Z=(\vec Z_0,...,\vec Z_m)$ be a random walk on $\mathbb{S}[G]$ induced by $\bx$, where $m\coloneqq n^{1/3}$. Let $\cE$ be the event that $Z$ is $(\mathbb{S}[G],\alpha)$-well-behaved. Then \mbox{$\pr[\cE]\ge 1-n^{-1/4}$}.
\end{cor}

\begin{proof}
Using Corollary~\ref{cor: a random walk is likely (U,a,a')-expected} and Lemma~\ref{lemma: a random walk is likely self-avoiding}, we have
$$\pr[\cE]\ge 1-\frac{2m^2b^2}{n}-3|\cS(G)|\frac{1}{n^k}\ge 1-2b^2n^{-\frac{1}{3}}-3{n\choose k-1}\frac{1}{n^k}\ge1-n^{-\frac{1}{4}},$$
which completes the proof.
\end{proof}

\section{Counting Hamilton cycles}\label{section: Counting Hamilton cycles}

\subsection{Counting well-behaved random walks}
We now find a lower bound on the entropy of a random walk on $\mathbb{S}[G]$. This then provides a lower bound on the number of such random walks. For convenience, given an $(n,k,\delta)$-graph $G$, we define
$$\overline h(G)\coloneqq h(G)-\frac{n}{k}\log{n\choose k-1}+\frac{n}{k}\log n,$$
since this expression appears frequently in this section. Note that \mbox{$\overline h(G)=\Theta(n\log n)$}.

\begin{lemma}\label{lemma: entropy of a random walk}
Suppose $1/n\ll \hat\delta,1/k, 1/b,\eps$. Let $G$ be an $(n,k,\delta)$-graph and $\bx$ be a $b$-normal perfect fractional matching of $G$ such that $h(\bx)\ge h(G)-\eps n$. Fix $\vec T\in \vec{\cS}(G)$. Let $Z=(\vec Z_0,...,\vec Z_m)$ be a random walk on $\mathbb{S}[G]$ induced by $\bx$ such that $\vec Z_0=\vec T$ and $m\coloneqq n^{1/3}$. Then $H(Z)\ge \frac{m}{n}k\overline h(G)-2k\eps m$.
\end{lemma}

\begin{proof}
We use the chain rule and Corollary~\ref{cor: rapid mixing in S[G]} to obtain the desired result. Choose $\alpha$ with $1/n\ll \alpha\ll\hat\delta,1/k, 1/b$ such that Corollary~\ref{cor: rapid mixing in S[G]} holds. Let $a\coloneqq\ln n+1>\alpha^{-2}$. Using the definition of a Markov chain, for $a\le i\le m$, we have
\begin{align*}
H(\vec Z_i|\vec Z_0,...,\vec Z_{i-1})&=H(\vec Z_i|\vec Z_{i-1})=\sum_{\vec S\in \vec{\cS}(G)}\pr[\vec Z_{i-1}=\vec S]H(\vec Z_i|{\vec Z_{i-1}=\vec S})\\
&\ge\sum_{\vec S\in \vec{\cS}(G)}\left(1-e^{-\alpha\ln n}\right)\frac{\bx_S}{(k-1)!n}\cdot h_{\bx}(S)\\
&\ge\left(1-n^{-\alpha}\right)\frac{1}{n}\sum_{S\in \cS(G)}\bx_Sh_{\bx}(S).
\end{align*}
Using Lemma~\ref{lemma: S[G] property vi}, we have
\begin{align*}
&\mathrel{}\sum_{S\in \cS(G)}\bx_Sh_{\bx}(S)\ge\mathrel{} kh(\bx)-n\log{n\choose k-1}+n\log n\ge k\overline h(G)-k\eps n.
\end{align*}
Using the chain rule of entropy (Fact~\ref{fact: CR for entropy}), we have
\begin{align*}
H(Z)&=\sum_{i=1}^mH(\vec Z_i|\vec Z_0,...,\vec Z
_{i-1})\ge (m-a+1)\cdot\left(1-n^{-\alpha}\right)\frac{1}{n}\left(k\overline h(G)-k\eps n\right)\\
&\ge\left(1-n^{-\frac{1}{3}}\ln n-n^{-\alpha}\right)\frac{m}{n}\left(k\overline h(G)-k\eps n\right)\ge \frac{m}{n}k\overline h(G)-2k\eps m,
\end{align*}
which completes the proof.
\end{proof}

\begin{cor}\label{cor: counting well-behaved paths}
Suppose $1/n\ll\alpha\ll \hat\delta,1/k, 1/b,\eps $. Let $G$ be an $(n,k,\delta)$-graph and $\bx$ be a $b$-normal perfect fractional matching of $G$ such that $h(\bx)\ge h(G)-\eps n$. Fix $\vec S\in \vec{\cS}(G)$. Let $Z=(\vec Z_0,...,\vec Z_m)$ be a random walk on $\mathbb{S}[G]$ induced by $\bx$ such that $\vec Z_0=\vec S$ and $m\coloneqq n^{1/3}$. Let $\cZ$ be the set of $(\mathbb{S}[G], \alpha)$-well-behaved realizations of $Z$. Then $\log |\cZ|\ge \frac{m}{n}k\overline h(G)-3k\eps m$.
\end{cor}

\begin{proof}
Choose $\alpha$ with $1/n\ll\alpha\ll \hat\delta,1/k, 1/b$ such that Corollary~\ref{cor: a random walk is likely good} holds. Let $\cE$ denote the event that $Z$ is $(\mathbb{S}[G], \alpha)$-well-behaved. By Corollary \ref{cor: a random walk is likely good}, we have $\pr[\cE]\ge 1-n^{-1/4}$. Let $W$ be any realization of $Z$, then Lemma~\ref{lemma: S[G] property v} yields $\pr[Z=W]\ge (b^2n)^{-m}.$ By Facts~\ref{fact: max entropy} and~\ref{fact: entropy given a likely event} as well as Lemma~\ref{lemma: entropy of a random walk}, we have
$$\log|\cZ|\ge H(Z|\cE)\ge H(Z)-2n^{-\frac{1}{4}}\log (b^2n)^m\ge \frac{m}{n}k\overline h(G)-3k\eps m,$$
as desired.
\end{proof}

\subsection{Counting long paths}
The next lemma forms the basis for iterating Corollary~\ref{cor: counting well-behaved paths}. It states that if we remove a $(G, \alpha)$-well-behaved path from an $(n, k, \delta)$-graph~$G$, the resulting subgraph remains $\delta'$-Dirac for some $\delta'$ only slightly smaller than $\delta$. This allows us to apply Corollary~\ref{cor: counting well-behaved paths} again to the subgraph.

\begin{lemma}\label{lemma: graph remains Dirac after removal of a well-behaved path}
Suppose $1/n\ll\alpha\ll 1/b\ll\eps \ll\hat\delta, 1/k$. Let $G$ be an $(n,k,\delta)$-graph and $\bx$ be a $b$-normal perfect fractional matching of $G$ such that $h(\bx)\ge h(G)-\eps n$. Suppose $P=v_1...v_m$ is a $(G, \alpha)$-well-behaved path, where $m\coloneqq n^{1/3}$. Let $G'\coloneqq G-P$, $n'\coloneqq n-m$, $\delta'\coloneqq \delta -2n^{-\frac{2}{3}-\alpha}$ and $b'\coloneqq (1+n^{-\frac{2}{3}-\frac{\alpha}{2}})b$. Then $G'$ is an $(n', k, \delta')$-graph, and there exists a $b'$-normal perfect fractional matching~$\bx'$ of $G'$ such that $h(\bx')\ge h(G')-\eps n'$ and
\begin{equation}\label{eq: condition on x'}
h(\bx')\ge \frac{n'}{n}h(\bx)-\frac{k-1}{k}n'\log \frac{n}{n'}-n^{\frac{1}{3}-\frac{\alpha}{2}}.
\end{equation}
\end{lemma}

\begin{proof}
Choose $\eps $ and $b\ge 4$ such that Lemma~\ref{lemma: existence of b-normal pfm} holds for the parameters $\frac{\hat\delta}{2}$ (instead of $\hat\delta$) and $k$. Choose $\alpha$ such that  Lemma~\ref{lemma: entropy is as expected after removing a well-behaved subset} holds. Since $P$ is $(G, \alpha)$-well-behaved, for each $S\in \cS(G')$, we have
\begin{align*}
|N_{G'}(S)|&= |N_G(S)|-|N_G(S)\cap P|\ge |N_G(S)|-\frac{m}{n}|N_G(S)|-n^{\frac{1}{3}-\alpha}\\
&\ge\frac{n'}{n}\delta n-\frac{n^{\frac{1}{3}-\alpha}}{n-m}n'\ge\left(\delta-2n^{-\frac{2}{3}-\alpha}\right)n'>\frac{1}{2}n'.
\end{align*}
Thus, $G'$ is an $(n',k,\delta')$-graph. By Lemma~\ref{lemma: entropy is as expected after removing a well-behaved subset}, there exists a $b'$-normal perfect fractional matching $\bz$ of $G'$ satisfying (\ref{eq: condition on x'}). By the choice of $b$ and Lemma~\ref{lemma: existence of b-normal pfm}, there exists a $b$-normal perfect fractional matching $\bz'$ of $G'$ such that \mbox{$h(\bz')\ge h(G')-\eps n'$}.  Since $b'>b$, taking $\bx'$ to be the one between $\bz$ and $\bz'$ that achieves the higher entropy completes the proof.
\end{proof}

\begin{lemma}\label{lemma: counting long paths}
Suppose $1/n\ll \beta\ll\eps\ll\hat\delta, 1/k$. Let $G$ be an $(n,k,\delta)$-graph. Fix $\vec S\in\vec{\cS}(G)$. Let $z$ be the number of paths in $G$ that start at $\vec S$ and have length at least $n-n^{1-\beta}$. Then $\log z\ge k\overline h(G)-n\log e-\eps n.$
\end{lemma}

\begin{proof}
Choose $\eps$ and $b\ge 4$ with $\beta\le 1/b\le\eps$ such that Lemma~\ref{lemma: existence of b-normal pfm} holds for the parameters $\frac{\hat\delta}{2}$ (instead of $\hat\delta$) and $k$. Choose $\alpha\coloneqq 3\beta$ so that Lemma~\ref{cor: counting well-behaved paths} and Lemma~\ref{lemma: graph remains Dirac after removal of a well-behaved path} hold for the parameters $\frac{\hat\delta}{2}$ (instead of $\hat\delta$), $k,\eps $ and $2b$ (instead of $b$). By Lemma~\ref{lemma: existence of b-normal pfm}, $G$ admits a $b$-normal perfect fractional matching $\bx$ with $h(\bx)\ge h(G)-\eps n$. We claim that $\delta_i\ge\frac{1}{2}+\frac{\hat\delta}{2}$ and $b_i\le 2b$ in the following procedure. Let $G_0\coloneqq G$, $n_0\coloneqq n$, $\delta_0\coloneqq\delta$, $b_0\coloneqq b$, $\bx_0\coloneqq \bx$ and $\vec S_0\coloneqq \vec S$. We construct a path in $G$ as follows.

At each step $i\ge 0$, let $\cZ_{i}$ be the set of all $(\mathbb{S}[G_i],\alpha)$-well-behaved realizations of the random walk $Z=(\vec Z_0,...,\vec Z_{m_i})$ on $\mathbb{S}[G_i]$ induced by $\bx_i$, where $\vec Z_0=\vec S_i$ and $m_i\coloneqq n_i^{1/3}$. Fix some walk $Z_i\in \cZ_i$, and let $P_i$ be the path in $G_i$ consisting of the first $m_i$ vertices in $Z_i$, then $P_i$ is a $(G_i,\alpha)$-well-behaved path. Let $\vec S_{i+1}$ be the end of $Z_i$. Define
\begin{align*}
G_{i+1}&\coloneqq G_i-P_i,\\
n_{i+1}&\coloneqq n_i-m_i,\\
\delta_{i+1}&\coloneqq \delta_i-2n_i^{-\frac{2}{3}-\alpha},\\
b_{i+1}&\coloneqq \left(1+n_i^{-\frac{2}{3}-\frac{\alpha}{2}}\right)b_i,\text{ and}\\
z_i&\coloneqq \min|\cZ_i|,
\end{align*}
where the minimum is taken over all possible paths $P_0\cup...\cup P_{i-1}$ and starting points $\vec S_i$. By Lemma~\ref{lemma: graph remains Dirac after removal of a well-behaved path}, $G_{i+1}$ is an $(n_{i+1},k,\delta_{i+1})$-graph, and there exists a $b_{i+1}$-normal perfect fractional matching $\bx_{i+1}$ of $G_{i+1}$ such that $h(\bx_{i+1})\ge h(G_{i+1})-\eps n_{i+1}$ and 
\begin{equation}\label{eq: update h(x)}
h(\bx_{i+1})\ge\frac{n_{i+1}}{n_i}h(\bx_i)-\frac{(k-1)}{k}n_{i+1}\log \frac{n_i}{n_{i+1}}-n_i^{\frac{1}{3}-\frac{\alpha}{2}}.
\end{equation}

We terminate the procedure once $n_i\le n^{1-\beta}.$ Let $\kappa$ be the step such that $n_{\kappa+1}\le n^{1-\beta}<n_{\kappa}$. Note that $Q\coloneqq P_0\cup...\cup P_{\kappa}$ is a path of length at least $n-n^{1-\beta}$. This procedure is finite since at every step $i$, $G_{i+1}$ is a proper subgraph of $G_i$. Moreover, since $m_0\ge...\ge m_{\kappa-1}$ and $m_0+...+m_{\kappa-1}\le n$, we have
$$\kappa\le \frac{n}{m_{\kappa-1}}\le \frac{n}{n_{\kappa-1}^{1/3}}<\frac{n}{(n^{1-\beta})^{\frac{1}{3}}}\le n^{\frac{2}{3}+\frac{\beta}{3}}.$$
We first compute a uniform bound for the $\delta_i$'s and the $b_i$'s. For all $i\in [\kappa]_0$, we have
\begin{align}
\delta_i&\ge \delta_0-2in_0^{-\frac{2}{3}-\alpha}\ge \delta-2\kappa n^{-\frac{2}{3}-\alpha}\ge \delta-2n^{\frac{2}{3}+\frac{\beta}{3}}n^{-\frac{2}{3}-\alpha}\ge \delta-2n^{\frac{\alpha}{9}-\alpha}\ge \frac{1}{2}+\frac{\hat\delta}{2}>\frac{1}{2},\text{ and}\label{eq: delta_i bound}\\
b_i&\le\left(1+n_{\kappa}^{-\frac{2}{3}-\frac{\alpha}{2}}\right)^i b_0\le \left(1+n^{(1-\beta)(-\frac{2}{3}-\frac{\alpha}{2})}\right)^{\kappa}b\le \left(1+n^{-\frac{2}{3}-\frac{\alpha}{4}}\right)^{n^{\frac{2}{3}+\frac{\beta}{3}}}b\le 2b.\label{eq: b_i bound}
\end{align}
The last inequality follows from
$$\log \left(1+n^{-\frac{2}{3}-\frac{\alpha}{4}}\right)^{n^{\frac{2}{3}+\frac{\beta}{3}}}=n^{\frac{2}{3}+\frac{\beta}{3}}\log \left(1+n^{-\frac{2}{3}-\frac{\alpha}{4}}\right)\le 2n^{\frac{2}{3}+\frac{\alpha}{9}}n^{-\frac{2}{3}-\frac{\alpha}{4}}\le 1.$$
Thus, for each $i\in [\kappa]_0$, $G_i$ is an $(n_i, k,\frac{1}{2}+\frac{\hat\delta}{2})$-graph and $\bx_i$ is a $2b$-normal perfect fractional matching of $G_i$. We now show by induction that
\begin{equation}\label{eq: induction}
h(G_i)\ge h(\bx_i)\ge \frac{n_i}{n}h(\bx)-\frac{k-1}{k}n_i\log \frac{n}{n_i}-in^{\frac{1}{3}-\frac{\alpha}{2}}.
\end{equation}
The base case $i=0$ is immediate. For $i\in[\kappa]$, starting with~(\ref{eq: update h(x)}) and followed by the induction hypothesis, we have
\begin{align*}
h(\bx_i)&\ge \frac{n_{i}}{n_{i-1}}h(\bx_{i-1})-\frac{k-1}{k}n_i\log \frac{n_{i-1}}{n_i}-n_{i-1}^{\frac{1}{3}-\frac{\alpha}{2}}\\
&\ge\frac{n_{i}}{n_{i-1}}\left(\frac{n_{i-1}}{n}h(\bx)-\frac{k-1}{k}n_{i-1}\log \frac{n}{n_{i-1}}-(i-1)n^{\frac{1}{3}-\frac{\alpha}{2}}\right)-\frac{k-1}{k}n_i\log \frac{n_{i-1}}{n_i}-n^{\frac{1}{3}-\frac{\alpha}{2}}\\
&\ge\frac{n_i}{n}h(\bx)-\frac{k-1}{k}n_i\log \frac{n}{n_i}-in^{\frac{1}{3}-\frac{\alpha}{2}}.
\end{align*}
This proves~(\ref{eq: induction}). Next, we utilize~(\ref{eq: induction}) to bound $\log z_i$ from below. By Corollary~\ref{cor: counting well-behaved paths} and equation~(\ref{eq: induction}), we have
\begin{align*}
\log z_i&\ge k\frac{m_i}{n_i}h(G_i)-m_i\log{n_i\choose k-1}+m_i\log n_i-3k\eps m_i\\
&\ge k\frac{m_i}{n}h(\bx)-(k-1)m_i\log \frac{n}{n_i}-ik\frac{m_i}{n_i}n^{\frac{1}{3}-\frac{\alpha}{2}}-m_i\log{n_i\choose k-1}+m_i\log n_i-3k\eps m_i.
\end{align*}
We consider the following sums over $i\in[\kappa]_0$. By construction, $\sum_{i=0}^{\kappa}m_i=n-n_{\kappa+1}\ge n-n^{1-\beta}$.
Using Stirling's approximation, we obtain
\begin{align*}
\sum_{i=0}^{\kappa}m_i\log n_i&=\sum_{i=0}^{\kappa}\log n_i^{m_i}\ge \sum_{i=0}^{\kappa}\log \frac{n_i!}{(n_i-m_i)!}=\sum_{i=0}^{\kappa}(\log n_i!-\log n_{i+1}!)=\log n!-\log n_{\kappa+1}!\\
&\ge n\log\frac{n}{e}-\log n -n_{\kappa+1}\log\frac{n_{\kappa+1}}{e}-\log n_{\kappa+1}\ge n\log n-n\log e-\frac{\eps}{2}n.\addtocounter{equation}{1}\tag{\theequation}\label{eq: sum 1}
\end{align*}
Moreover,
\begin{equation}\label{eq: sum 2}
\sum_{i=0}^{\kappa}i\frac{m_i}{n_i}n^{\frac{1}{3}-\frac{\alpha}{2}}\le\sum_{i=0}^{\kappa}i\frac{1}{n_{\kappa}^{2/3}
}n^{\frac{1}{3}-\frac{\alpha}{2}}\le\frac{\kappa^2}{(n^{1-\beta})^{2/3}
}n^{\frac{1}{3}-\frac{\alpha}{2}}\le n^{\frac{4}{3}+\frac{2\beta}{3}}n^{-\frac{2}{3}+\frac{2\beta}{3}}n^{\frac{1}{3}-\frac{\alpha}{2}}\le\frac{\eps}{2}n.
\end{equation}
Note that for $i\in[\kappa]$, we have
$$(k-1)\log \frac{n}{n_i}+\log{n_i\choose k-1}\le\mathrel{}\log\left[\frac{n}{n_i}\cdot\frac{n-1}{n_i-1}\cdot...\cdot\frac{n-k+2}{n_i-k+2}\cdot{n_i\choose k-1}\right]=\log{n\choose k-1}.$$
Let $z$ denote the number of paths in $G$ that start at $\vec S$ and have length at least $n-n^{1-\beta}$. Note that every path $Q$ constructed above is such a path. Thus, we have
\begin{align*}
\log z&\ge \log \prod_{i=0}^{\kappa} z_i=\sum_{i=0}^{\kappa}\log z_i\\
&\ge\sum_{i=0}^{\kappa} k\frac{m_i}{n}h(\bx)-m_i\log {n\choose k-1}+m_i\log n_i-ik\frac{m_i}{n_i}n^{\frac{1}{3}-\frac{\alpha}{2}}-3k\eps m_i\\
&\ge k(h(G)-\eps n)-n\log{n\choose k-1}+n\log n-n\log e-\eps n-3k\eps n\\
&\ge k\overline h(G)-n\log e-5k\eps n.
\end{align*}
This completes the proof.
\end{proof}

\subsection{Absorbing long paths}
Given an $(n,k,\delta)$-graph $G$, the following two lemmas enable us to complete the paths from Lemma~\ref{lemma: counting long paths} into a Hamilton cycle of $G$. Lemma~\ref{lemma: absorption subset} guarantees the existence of a vertex subset $U\subseteq V(G)$ that contains sufficiently many neighbours of every $(k-1)$-subset of $V(G)$. Applying Lemma~\ref{lemma: counting long paths}, we find many long paths in $G-U$. Then Lemma~\ref{lemma: tight Hamilton-connected} allows us to complete each of these paths into a Hamilton cycle of $G$.

\begin{lemma}\label{lemma: absorption subset}
Suppose $1/n\ll\hat\delta,1/k$. Let $G$ be an $(n,k,\delta)$-graph. Then for all $0<\alpha\le1$, there exists a set $U\subseteq V(G)$ of size $n^{\alpha}$ such that for all $S\in \cS(G)$, we have
\begin{equation}\label{eq: absorption subset}
|N_G(S)\cap U|\ge\left(\frac{1}{2}+\frac{3\hat\delta}{4}\right)|U|.
\end{equation}
\end{lemma}

We omit the standard proof. It follows directly by choosing a set $U$ of size $n^{\alpha}$ uniformly at random and applying Lemma~\ref{lemma: chernoff}.

\begin{lemma}[{\cite[Lemma~3.7]{GGJKO:21}}]\label{lemma: tight Hamilton-connected}
Suppose $1/n\ll\hat\delta,1/k$. Let $G$ be an $(n,k,\delta)$-graph. Then for any disjoint $\vec S,\vec T\in\vec{\cS}(G)$, there is a Hamilton path of $G$ from $\vec S$ to $\vec T$.
\end{lemma}

We are now able to prove our main theorem.

\begin{proof}[Proof of Theorem~\ref{main theorem - with entropy}]
Choose $\beta$ and $\eps$ such that Lemma~\ref{lemma: counting long paths} holds for the parameters $\frac{\hat\delta}{2}$ (instead of $\hat\delta$) and $k$. Let $U\subseteq V(G)$ be a set of size $n^{\alpha}$ guaranteed by Lemma~\ref{lemma: absorption subset} with $\alpha\coloneqq 1-\beta/2$. Let $G'\coloneqq G-U$ and $n'\coloneqq n-n^{\alpha}$. Then $G'$ is an $(n', k, \frac{1}{2}+\frac{\hat\delta}{2})$-graph since
$$\delta_{k-1}(G')\ge \delta_{k-1}(G)-|U|\ge\delta n-n^{\alpha}\ge\left(\frac{1}{2}+\frac{\hat\delta}{2}\right)n'.$$
By Lemma~\ref{lemma: entropy decreases slightly after removing a small subset}, we have $h(G')\ge h(G)-\eps n$. Fix $\vec S\in \vec{\cS}(G')$. Let $Q$ be a path in $G'$ of length at least $n'-(n')^{1-\beta}$ from $\vec S$ to $\vec T$ for some $\vec T\in\vec{\cS}(G')$. Let $G''\coloneqq G[V(G'-Q)\cup U\cup S\cup T]$, and $n''\coloneqq |V(G'')|\le (n')^{1-\beta}+n^{\alpha}+2(k-1)\le n^{\alpha}+n^{1-\beta}$. Note that $G''$ is a $(n'', k, \frac{1}{2}+\frac{\hat\delta}{2})$-graph since for all $S\in \cS(G'')\subseteq \cS(G)$, we have
$$|N_{G''}(S)|\ge |N_G(S)\cap U|\ge \left(\frac{1}{2}+\frac{3\hat\delta}{4}\right)n^{\alpha}\ge\left(\frac{1}{2}+\frac{\hat\delta}{2}\right)(n^{\alpha}+n^{1-\beta})\ge\left(\frac{1}{2}+\frac{\hat\delta}{2}\right)n''.$$
By Lemma~\ref{lemma: tight Hamilton-connected}, there exists a Hamilton path $P$ from $\vec T$ to $\vec S$ in $G''$. Then $P+Q$ is a Hamilton cycle of $G$. Since $Q$ was arbitrary, by Lemma~\ref{lemma: counting long paths}, we have
\begin{align*}
\log \Psi(G)&\ge kh(G')-n'\log{n'\choose k-1}+n'\log n'-n'\log e-\eps n'\\
&\ge k(h(G)-\eps n)-n\log{n\choose k-1}+n\log n-n\log e-2\eps n\\
&\ge kh(G)-n\log{n\choose k-1}+n\log n-n\log e-2k\eps n.
\end{align*}
This completes the proof.
\end{proof}

Theorem~\ref{main theorem - without entropy} follows as a corollary.
\begin{proof}[Proof of Theorem~\ref{main theorem - without entropy}]
We use Theorem~\ref{main theorem - with entropy} together with the lower bound for $h(G)$ given by Theorem~\ref{thm: graph entropy bound} to obtain
$$\log \Psi(G)\ge n\log\delta+n\log n-n\log e-o(n)\ge n\log \delta+\log (n!)-o(n).$$
Therefore, we conclude $\Psi(G)\ge (\delta-o(1))^n n!$.
\end{proof}

\section{Conclusion}

In this paper we established a tight lower bound for the number of tight Hamilton cycles in $k$-graphs with large minimum $(k-1)$-degree, thereby solving a conjecture of Ferber, Hardiman and Mond in a strong form.

Proving that our bound is in fact tight up to the error term in the exponent for \emph{all} $k$-graphs~$G$ with $\delta_{k-1}(G)\geq (1/2+o(1))n$ is an obvious interesting open problem (this is true for graphs).
Observe that for graphs, a tight upper bound on the number of Hamilton cycles $\Psi(G)$ can be easily deduced from a tight upper bound on the number of perfect matchings $\Phi(G)$ by noting that $\Psi(G)\leq (\Phi(G))^2$.
On the exponential scale, this is tight because it is not so unlikely that two randomly chosen perfect matchings form a Hamilton cycle in graphs; whereas in $k$-graphs, $k$ randomly chosen perfect matchings form a tight Hamilton cycle only with exponentially small probability.

\bibliographystyle{amsplain}
\providecommand{\bysame}{\leavevmode\hbox to3em{\hrulefill}\thinspace}
\providecommand{\MR}{\relax\ifhmode\unskip\space\fi MR }
\providecommand{\MRhref}[2]{%
	\href{http://www.ams.org/mathscinet-getitem?mr=#1}{#2}
}
\providecommand{\href}[2]{#2}

\section*{Appendix}

We now extend our method in order to count the number of Hamilton $\ell$-cycles in $k$-graphs for $\ell\in[k-1]_0$, thereby proving Theorem~\ref{main theorem - with entropy - non-tight}. Sections~\ref{section: Preliminaries - non-tight},~\ref{section: Probabilistic properties of random ell-walks} and~\ref{section: Counting Hamilton ell-cycles} generalizes the definitions and results in Sections~\ref{section: Preliminaries},~\ref{section: Probabilistic properties of random walks} and~\ref{section: Counting Hamilton cycles}, respectively.

\section{Preliminaries}\label{section: Preliminaries - non-tight}

Let $G$ be a $k$-graph. For $\ell\in[k-1]_0$, we say a subgraph $W$ of $G$ is an \emph{$\ell$-walk} of length $m$ if $W$ admits an ordering $\sigma$ of its vertices $V(W)=(v_1, ...,v_{(k-\ell)m+\ell})$ such that $E(W)=\{\{v_{i(k-\ell)+1},..., v_{i(k-\ell)+k}\}:i\in [m-1]_0\}$. When the ordering $\sigma$ is specified, the walk $W$ is said to be \emph{ordered}. The walk~$W$ is an \mbox{\emph{$\ell$-path}} if the vertices $v_1,...,v_{(k-\ell)m+\ell}$ are distinct. The walk~$W$ is an \emph{$\ell$-cycle} if $(v_1,...,v_{\ell})=(v_{(k-\ell)m+1},...,v_{(k-\ell)m+\ell})$ and the vertices $v_1,...,v_{(k-\ell)m}$ are distinct. A \emph{Hamilton $\ell$-path} in $G$ is an $\ell$-path~$P$ such that $V(P)=V(G)$. A \emph{Hamilton $\ell$-cycle} in $G$ is an $\ell$-cycle~$C$ such that $V(C)=V(G)$. When $\ell\in[k-1]$ and $W$ is a Hamilton $\ell$-cycle, we consider orderings up to reversals and cyclic shifts by multiples of $k-\ell$. That is, two orderings $\sigma$ and $\sigma'$ of $V(W)$ define the same ordered Hamilton $\ell$-cycle if one can be obtained from the other via a cyclic shift by a mutiple of $k-\ell$ and/or reversing the order.

Given an $(n,k,\delta)$-graph $G$ with $\ell\in[k-1]_0$ and $(k-\ell)|n$, let $\Psi_{\ell}(G)$ be the number of Hamilton \mbox{$\ell$-cycles} in $G$. Let $\Psi_{\ell}'(G)$ be the number of ordered Hamilton $\ell$-cycles in $G$. Observe that
$\Psi_{\ell}'(G)=C_{n,k,\ell}\Psi_{\ell}(G)$.

\begin{cor}\label{cor: lower bound on minimum d-degree}
Let $G$ be an $(n,k,\delta)$-graph. Then for all $d\in[k-1]_0$ and sufficiently large $n$, we have
$$\delta_d(G)\ge\frac{\delta n{n-d\choose k-d}}{n-k+1}\ge\frac{\delta}{k^k}n^{k-d}.$$
\end{cor}

\begin{proof}
By Fact~\ref{lemma: monotonicity of minimum degrees}, we have
\begin{align*}
\delta_d(G)&\ge\delta_{k-1}(G)\frac{{n-d\choose k-d}}{{n-k+1\choose 1}}\ge\delta n\frac{1}{n-k+1}\left(\frac{n-d}{k-d}\right)^{k-d}\ge\frac{\delta}{k^k}n^{k-d}.
\end{align*}
\end{proof}

\subsection{Auxiliary graph}\label{section: auxiliary graph - non-tight}
Let $G$ be an $(n,k,\delta)$-graph and $\bx$ be a perfect fractional matching of $G$. We now define an auxiliary graph on $E(G)$, which conceptually generalizes the graph $\mathbb{S}[G]$ defined in Section~\ref{section: auxiliary graph}. Let $\ell\in[k]_0$. Define \mbox{$\cS_{\ell}(G)\coloneqq {V(G)\choose\ell}$} to be the set of all $\ell$-subsets of $V(G)$. Let $\vec{\cS}_{\ell}(G)$ denote the set of all $\ell$-tuples in $V(G)^{\ell}$ with distinct entries. For each \mbox{$\vec S=(s_1,...,s_{\ell})\in\vec{\cS}_{\ell}(G)$}, we write $S=\{s_1,...,s_{\ell}\}$ for the underlying $\ell$-set in $\cS_{\ell}(G)$. Recall that for each $S\in\cS_{\ell}(G)$, we define \mbox{$\bx_ S= \sum_{X\in N_G(S)}\bx[S\cup X]$}. We also write $\bx[S]$ for $\bx_S$. We define the set of ordered edges $\vec{E}(G)\coloneqq\{\vec e\in \vec{\cS}_k(G)|e\in E(G)\}$. For any tuple $\vec S=(s_1,...,s_m)$ and $r\in [m]$, we define $\vec S^{(r)-}\coloneqq (s_1,...,s_r)$ and $\vec S^{(r)+}\coloneqq(s_{k-r+1},...,s_k)$. We also use $S^{(r)-}$ and $S^{(r)+}$ to denote the corresponding $r$-sets.

We defined $\mathbb{S}_{\ell}[G]$ to be the weighted digraph with vertex set $\vec E(G)$, where for all $\vec e=(e_1,...,e_k), \vec f=(f_1,...,f_k)\in \vec E(G)$, there is a directed edge $\vec e\vec f\in E(\mathbb{S}_{\ell}[G])$ if $\vec e^{(\ell)+}=\vec f^{(\ell)-}$. The edge weighting $\by:E(\mathbb{S}_{\ell}[G])\to[0,1]$ in $\mathbb{S}[G]$ induced by $\bx$ is defined by
$$\by\left[\vec e\vec f\right]\coloneqq\begin{cases}\frac{\bx_f}{(k-\ell)!\bx[e^{(\ell)+}]}&\text{if }\bx[e^{(\ell)+}]\neq 0\\
\frac{1}{(k-\ell)!\deg_G(e^{(\ell)+})}&\text{if }\bx[e^{(\ell)+}]=0\end{cases}$$
for all $\vec e\vec f\in E(\mathbb{S}_{\ell}[G])$. When $\bx$ is $b$-normal for some \mbox{$b\ge 1$}, we always have $\bx[e^{(\ell)+}]\neq 0$. Moreover, for simplicity, we often write \mbox{$\by[e^{(\ell)+}\to f^{(k-\ell)+}]$} for \mbox{$(k-\ell)!\by[\vec e\vec f]$}. This means that when $\bx[e^{(\ell)+}]\neq 0$, for all \mbox{$X\in N_G(e^{(\ell)+})$}, we have $\by[e^{(\ell)+}\to X]=\frac{\bx[e^{(\ell)+}\cup X]}{\bx[e^{(\ell)+}]}$. We further observe that for each $\vec e\in \vec E(G)$, we have \mbox{$\sum_{\vec f\in N^+_{\mathbb{S}_{\ell}[G]}(\vec e)}\by[\vec e\vec f]=\sum_{X\in N_G(e^{(\ell)+})}\by[e^{(\ell)+}\to X]=1$}. So we may interpret $\by$ as a probability distribution on the outgoing edges of $\vec e\in \vec E(G)$.

For each $S\in \cS_{\ell}(G)$, we define the \emph{entropy of $S$ induced by $\bx$} by
$$h_{\bx}(S)\coloneqq \sum_{X\in N_G(S)}\by[S\to X]\log\frac{(k-\ell)!}{\by[S\to X]}.$$

Note that for any $\vec e\in\vec E(G)$, $h_{\bx}(e^{(\ell)+})=\sum_{\vec f\in N_{\mathbb{S}_{\ell}(G)}^+(\vec e)}\by[\vec e\vec f]\log\frac{1}{\by[\vec e\vec f]}$. There is a natural bijection between the directed walks in $\mathbb{S}_{\ell}[G]$ and the ordered $\ell$-walks in $G$, by tracing the same sequence of vertices in $G$.

We collect some useful properties in the following lemma.

\begin{lemma}
Suppose $n\ge k\ge 2$. Let $\ell\in [k-1]_0$. Let $G$ be an $(n,k,\delta)$-graph and $\bx$ be a perfect fractional matching of $G$. Then the following holds.
\begin{enumerate}[label=\normalfont(\roman*), ref={\thelemma(\roman*)}]
\item $\sum_{S\in \cS_{\ell}(G)}\bx_S={k\choose \ell}\frac{n}{k}$, and\label{lemma: S[G] property ii - non-tight}
\item $\sum_{v\in S\in \cS_{\ell}(G)}\bx_S={k-1\choose\ell-1}$ for every vertex $v\in V(G)$.\label{lemma: S[G] property iii - non-tight}
\end{enumerate}
If we further assume that $\bx$ is $b$-normal for some $b\ge1$, then for sufficiently large $n$,
\begin{enumerate}[label=\normalfont(\roman*), ref={\thelemma(\roman*)}]\setcounter{enumi}{2}
\item $\sum_{S\in\cS_{\ell}(G)}\bx_S h_{\bx}(S)\ge{k\choose\ell}h(\bx)-{k\choose\ell}\frac{n}{k}\log {n\choose k-1}\frac{(n-k+1)!}{(n-\ell)!n}$,\label{lemma: S[G] property vi - non-tight}
\item $\frac{\delta}{k^kbn^{\ell-1}}\le\bx_S\le\frac{b}{n^{\ell-1}}$\label{lemma: S[G] property iv - non-tight}
for all $S\in\cS_{\ell}(G)$, and
\item $\frac{1}{(k-\ell)!b^2n^{k-\ell}}\le\by[\vec e\vec f]\le\frac{k^kb^2}{(k-\ell)!\delta n^{k-\ell}}$
for all $\vec e\vec f\in E(\mathbb{S}_{\ell}[G])$.\label{lemma: S[G] property v - non-tight}
\end{enumerate}
\end{lemma}

\begin{proof}
(i) and (ii) follow from the following calculations.
\begin{align*}
\sum_{S\in \cS_{\ell}(G)}\bx_S&=\sum_{S\in \cS_{\ell}(G)}\sum_{X\in N_G(S)}\bx[S\cup X]={k\choose \ell}\sum_{e\in E(G)}\bx_e={k\choose \ell}\frac{n}{k},\text{ and}\\
\sum_{v\in S\in \cS_{\ell}(G)}\bx_S&=\sum_{v\in S\in \cS_{\ell}(G)}\sum_{X\in N_G(S)}\bx[S\cup X]={k-1\choose\ell-1}\sum_{v\in e\in E(G)}\bx_e={k-1\choose\ell-1}
\end{align*}
for every vertex $v\in V(G)$. Next, suppose $\bx$ is $b$-normal. To prove (iii), we calculate
\begin{align*}
\sum_{S\in\cS_{\ell}(G)}\bx_S h_{\bx}(S)&=\sum_{S\in\cS_{\ell}(G)}\bx_S\sum_{X\in N_G(S)}\by[S\to X]\log\frac{(k-\ell)!}{\by[S\to X]}\\
&=\sum_{S\in\cS_{\ell}(G)}\sum_{X\in N_G(S)}\bx[ S\cup X]\log\frac{(k-\ell)!\bx_S}{\bx[S\cup X]}\\
&={k\choose\ell}\sum_{e\in E(G)}\bx_e\log\frac{1}{\bx_e}+\sum_{S\in\cS_{\ell}(G)}\bx_S\log(k-\ell)!+\sum_{S\in\cS_{\ell}(G)}\bx_S\log\bx_S\\
&\ge{k\choose\ell}h(\bx)+{k\choose\ell}\frac{n}{k}\log(k-\ell)!-{k\choose\ell}\frac{n}{k}\log \frac{{n\choose \ell}}{{k\choose\ell}\frac{n}{k}}\\
&={k\choose\ell}h(\bx)-{k\choose\ell}\frac{n}{k}\log \frac{k{n\choose \ell}}{n{k\choose\ell}(k-\ell)!}\\
&={k\choose\ell}h(\bx)-{k\choose\ell}\frac{n}{k}\log {n\choose k-1}\frac{(n-k+1)!}{(n-\ell)!n},
\end{align*}
where the inequality follows from Jensen's inequality applied to the concave function $t\mapsto \log t$, as shown below.
\begin{align*}
\sum_{S\in\cS_{\ell}(G)}\bx_S\log \frac{1}{\bx_S}&\le \left(\sum_{S\in\cS_{\ell}(G)}\bx_S\right)\log\frac{\sum_{S\in\cS_{\ell}(G)}\bx_S\frac{1}{\bx_S}}{\sum_{S\in\cS_{\ell}(G)}\bx_S}={k\choose\ell}\frac{n}{k}\log \frac{{n\choose \ell}}{{k\choose\ell}\frac{n}{k}}.
\end{align*}
This proves (iii). For each $S\in \cS_{\ell}(G)$, by Corollary~\ref{cor: lower bound on minimum d-degree}, we have
$$\frac{\delta}{k^kbn^{\ell-1}}\le\frac{\delta}{k^k}n^{k-\ell}\frac{1}{bn^{k-1}}\le\delta_{\ell}(G)\frac{1}{bn^{k-1}}\le\sum_{X\in N_G(S)}\bx[S\cup X]\le {n\choose k-\ell}\frac{b}{n^{k-1}}\le\frac{b}{n^{\ell-1}}.$$
This proves (iv). For each $\vec e\vec f\in E(\mathbb{S}_{\ell}[G])$, we exploit the previous inequality and obtain
$$\frac{1}{(k-\ell)!b^2n^{k-\ell}}=\frac{1}{(k-\ell)!}\frac{1}{bn^{k-1}}\frac{n^{\ell-1}}{b}\le\by[\vec e\vec f]\le\frac{1}{(k-\ell)!}\frac{b}{n^{k-1}}\frac{k^kbn^{\ell-1}}{\delta}= \frac{k^kb^2}{(k-\ell)!\delta n^{k-\ell}},$$
which proves (v).
\end{proof}

We note that Lemma~\ref{lemma: entropy is as expected after removing a well-behaved subset} still holds if $m\coloneqq\lfloor n^{1/3}\rfloor(k-\ell)$, for any $\ell\in[k-1]_0$.

\section{Probabilistic properties of random $\ell$-walks}\label{section: Probabilistic properties of random ell-walks}

In this section, we study random $\ell$-walks in an $(n,k,\delta)$-graph and in its auxiliary graph defined in Section~\ref{section: auxiliary graph - non-tight}. As in the $\ell=k-1$ case, we model these walks as Markov chains and show that after sufficiently many steps, the probability of visiting any vertex is roughly uniform.

Let $G$ be an $(n,k,\delta)$-graph and $\bx$ be a perfect fractional matching of $G$. Let $\ell\in[k-1]_0$. Let $\mathbb{S}_{\ell}[G]$ be as defined in Section~\ref{section: auxiliary graph - non-tight} and $\by$ be the corresponding edge weighting induced by $\bx$. A \emph{random walk} on $\mathbb{S}_{\ell}[G]$ induced by $\bx$ is the Markov chain $(\vec Z_t)_{t\in\bN_0}$ with state space $\vec E(G)$ and a transition matrix $P=(p_{\vec e\vec f})_{\vec e, \vec f\in\vec E(G)}$, defined by
$$p_{\vec e\vec f}=\begin{dcases}
\by[\vec e\vec f]&\quad \text{if }\vec e\vec f\in E(\mathbb{S}_{\ell}[G])\\
0&\quad\text{otherwise}.
\end{dcases}$$
A \emph{random (ordered) $\ell$-walk} on $G$ induced by $\bx$ is a sequence $(R_t)_{t\in \bN_0}$ of random variables with sample space $V(G)$ such that if $\vec Z_t\coloneqq (R_{(k-\ell)t},...,R_{(k-\ell)t+k-1})$ for all $t\in \bN_0$, then $(\vec Z_t)_{t\in \bN_0}$ is a random walk on $\mathbb{S}_{\ell}[G]$ induced by $\bx$. We note that every random walk on $\mathbb{S}_{\ell}[G]$ yields a random $\ell$-walk on $G$.

\subsection{Convergence of random $\ell$-walks}
We need the following generalized version of Lemma~\ref{lemma: well-connectedness}. Recall that an $\ell$-walk of length $L$ has $\ell$ edges and $(k-\ell)L+\ell$ vertices.

\begin{cor}\label{cor: well-connectedness}
Suppose $1/n\ll \zeta\ll 1/L\ll \hat\delta,1/k$. Let $\ell\in [k-1]_0$. Let $G$ be an $(n,k,\delta)$-graph. For any $\vec S, \vec T\in\vec{\cS}_{\ell}(G)$, there are at least $\zeta n^{(k-\ell)L-\ell}$ $\ell$-walks of length $L$ from $\vec S$ to $\vec T$ in $G$.
\end{cor}

\begin{proof}
Take $\vec S,\vec T\in\vec \cS_{\ell}(G)$. There are $N\coloneqq (k-1-\ell)!{n-\ell\choose k-1-\ell}$ ways to extend $\vec S$ to some $\vec S'\in\cS_{k-1}(G)$ such that $\vec S'^{(\ell)-}=\vec S$. Similarly, there are $N$ ways to extend $\vec T$ to some $\vec T'\in\cS_{k-1}(G)$ such that $\vec T'^{(\ell)+}=\vec T$. By Lemma~\ref{lemma: well-connectedness}, we may choose $1/n\ll\zeta'\ll 1/L'\ll \hat\delta,1/k$ with $(k-\ell)|(L'-1)$ such that there are $\zeta' n^{L'-(k-1)}$ $(k-1)$-walks $W$ of length $L'$ from each $\vec S'$ to $\vec T'$. Let 
$$L=\frac{L'-1}{k-\ell}+1\quad\text{and}\quad\zeta=\left(\frac{(k-\ell-1)!}{(2(k-1-\ell))^{k-1-\ell}}\right)^2\zeta',$$
so that $L'+(k-1)=(k-\ell)L+\ell$. By considering a subset of edges, each $W$ contains an $\ell$-walk of length $L$ from $\vec S$ to $\vec T$. Thus, there are at least
$$N^2\zeta' n^{L'-(k-1)}\ge\left(\frac{(k-\ell-1)!n^{k-\ell-1}}{(2(k-1-\ell))^{k-1-\ell}}\right)^2\zeta' n^{(k-\ell)(L-1)+1-(k-1)}\ge\zeta n^{(k-\ell)L-\ell}$$
$\ell$-walks of length $L$ from $\vec S$ to $\vec T$.
\end{proof}

We now apply the rapid mixing property of Markov chains to the random walks in $\cS_{\ell}[G]$.

\begin{cor}\label{cor: rapid mixing in S[G] - non-tight}
Suppose $1/n\ll\alpha\ll\hat\delta,1/k,1/b$. Let $\ell\in[k-1]_0$. Let $G$ be an $(n,k,\delta)$-graph and $\bx$ be a $b$-normal perfect fractional matching of $G$. Let $(\vec Z_t)_{t\in\bN_0}$ be a random walk on $\mathbb{S}_{\ell}[G]$ induced by $\bx$. Then, for any ordered edge $\vec e\in\vec E(G)$ and $t\ge\alpha^{-2}$, we have
$$\pr[\vec Z_t=\vec e]=\left(1\pm e^{-\alpha t}\right)\frac{\bx_e}{(k-1)!n}.$$
\end{cor}

\begin{proof}
Using the same notation as the beginning of Section~\ref{section: Probabilistic properties of random ell-walks}, $(\vec Z_t)_{t\in\bN_0}$ is a Markov chain with state space $\vec E(G)$ and transition matrix $P=(p_{\vec e\vec f})_{\vec e, \vec f\in\vec E(G)}$. By Lemma~\ref{lemma: S[G] property v - non-tight}, we have the following bounds on the nonzero entries of $P$,
$$\frac{1}{(k-\ell)!b^2n^{k-\ell}}\le p_{\vec e\vec f}\le\frac{k^kb^2}{(k-\ell)!\delta n^{k-\ell}}.$$

By Corollary~\ref{cor: well-connectedness}, there exist \mbox{$1/n\ll\zeta\ll 1/L\ll\hat\delta, 1/k$} such that for any $\vec e,\vec f\in\vec E(G)$, the $k$-graph $G$ has at least $\zeta n^{(k-\ell)(L-1)-\ell}$ $\ell$-walks of length $L-1$ from $\vec e^{(\ell)+}$ to $\vec f^{(\ell)-}$, and hence $\mathbb{S}_{\ell}[G]$ has at least $\zeta n^{(k-\ell)(L-1)-\ell}$ walks of length $L$ from $\vec e$ to $\vec f$. This allows us to define a new Markov chain $(\vec Z_t')_{t\in\bN_0}$ on $\vec E(G)$ with a strictly positive transition matrix $Q\coloneqq P^L$. We have the following bounds on the entries of $Q=(q_{\vec e\vec f})_{\vec e, \vec f\in\vec E(G)}$,
\begin{align}
q_{\vec e\vec f}&\le n^{(k-\ell)(L-1)-\ell}\left(\frac{k^kb^2}{(k-\ell)!\delta n^{k-\ell}}\right)^L=\left(\frac{k^kb^2}{(k-\ell)!\delta}\right)^L n^{-k}\text{, and}\label{eq: Q upper bound}\\
q_{\vec e\vec f}&\ge \zeta n^{(k-\ell)(L-1)-\ell}\left(\frac{1}{(k-\ell)!b^2n^{k-\ell}}\right)^L=\zeta\left(\frac{1}{(k-\ell)!b^2}\right)^L n^{-k}.\label{eq: Q lower bound}
\end{align}
We define $\sigma\coloneqq (\sigma_{\vec e})_{\vec e\in\vec E(G)}$ by
$$\sigma_{\vec e}\coloneqq\frac{\bx_e}{(k-1)!n}$$
for each $\vec e\in\vec E(G)$. We claim that $\sigma$ is the stationary distribution of the Markov chain $(Z_t')_{t\in\bN_0}$.

Indeed, using Lemma~\ref{lemma: S[G] property i}, it is straightforward to verify that $\sigma$ is a probability distribution. It remains to show that $\sigma$ is a left eigenvector of $P$, and hence a left eigenvector of $Q$. For all $\vec f\in\vec E(G)$, we have
\begin{align*}
(\sigma P)_{\vec f}&=\sum_{\vec e\in\vec E(G)}\sigma_{\vec e}p_{\vec e\vec f}=\sum_{\vec e\in\vec E(G)}\sigma_{\vec e}\by[\vec e\vec f]\1[\vec e\vec f\in E(\mathbb{S}_{\ell}[G])]\\
&=\sum_{\vec e\in\vec E(G)}\frac{\bx_e}{(k-1)!n}\frac{\bx_f}{(k-\ell)!\bx[e^{(\ell)+}]}\1[\vec e\vec f\in E(\mathbb{S}_{\ell}[G])]\\
&=\frac{\bx_f}{(k-1)!n}\sum_{\vec e\in\vec E(G)}\frac{\bx_e\1 [\vec e\vec f\in E(\mathbb{S}_{\ell}[G])]}{(k-\ell)!\bx[f^{(\ell)-}]}\\
&=\sigma_{\vec f}\sum_{X\in N_G(f^{(\ell)-})}\frac{\bx[X\cup f^{(\ell)-}]}{\bx[f^{(\ell)-}]}=\sigma_{\vec f},
\end{align*}
which proves our claim.

Since $\bx$ is $b$-normal, for each $\vec e\in\vec E(G)$, we have
$$\frac{1}{b(k-1)!n^k}\le \sigma_{\vec e}\le \frac{b}{(k-1)!n^k}.$$
Together with~(\ref{eq: Q upper bound}) and~(\ref{eq: Q lower bound}), this allows us to bound $c_1$ and $c_2$ as defined in Lemma~\ref{lemma: rapid mixing}. We have
\begin{align*}
c_1&\coloneqq \inf_{\vec e, \vec f, \vec g\in\vec E(G)}\frac{q_{\vec e\vec f}}{\sigma_{\vec g}}\ge\zeta\left(\frac{1}{(k-\ell)!b^2}\right)^L\cdot n^{-k}\cdot\frac{(k-1)!n^k}{b}=\zeta\left(\frac{1}{(k-\ell)!b^2}\right)^L\frac{(k-1)!}{b},\text{ and}\\
c_2&\coloneqq \sup_{\vec e, \vec f, \vec g\in\vec E(G)}\frac{q_{\vec e\vec f}}{\sigma_{\vec g}}\le\left(\frac{k^kb^2}{(k-\ell)!\delta}\right)^L\cdot n^{-k}\cdot b(k-1)!n^k=\left(\frac{k^kb^2}{(k-\ell)!\delta}\right)^L b(k-1)!.
\end{align*}
Let
$$\alpha\coloneqq\min\left\{\frac{\zeta}{2L}\left(\frac{1}{(k-\ell)!b^2}\right)^L\frac{(k-1)!}{b},\left(\frac{(k-\ell)!\delta}{k^kb^2}\right)^L\frac{1}{b(k-1)!}\right\}$$
so that $c_1\ge 2L\alpha$ and $c_2\le\alpha^{-1}$.
Using the inequality $(1-x)^t\le e^{-xt}$, Lemma~\ref{lemma: rapid mixing} yields
$$\pr[\vec Z_t=\vec e]=\pr[\vec Z_{t/L}'=\vec e]=\left(1\pm\left(1-\frac{c_1}{2}\right)^{\frac{t}{L}}\right)\sigma_{\vec e}=\left(1\pm\exp\left(-\frac{c_1 t}{2L}\right)\right)\sigma_{\vec e}=\left(1\pm e^{-\alpha t}\right)\sigma_{\vec e}$$
for any $\vec e\in\vec E(G)$ and any $t/L\ge 2+2(2L\alpha)^{-1}\log(\alpha^{-1})$ such that $L|t$. Similarly, we conclude that under the same conditions on $t$, we have
$$\pr[\vec Z_{i+t}=\vec e|\vec Z_i=\vec f]=\left(1\pm e^{-\alpha t}\right)\sigma_{\vec e}$$
for any $\vec e, \vec f\in\vec E(G)$. Therefore, we obtain
$$\pr[\vec Z_t=\vec e]=\sum_{\vec f\in\vec E(G)}\pr[\vec Z_{t\bmod L}=\vec f]\pr[\vec Z_t=\vec e|\vec Z_{t\bmod L}=\vec f]=\left(1\pm e^{-\alpha t}\right)\sigma_{\vec e}$$
for any $\vec e\in\vec E(G)$ and any $t\ge L(2+2(2L\alpha)^{-1}\log(\alpha^{-1}))+L-1=3L-1+\alpha^{-1}\log(\alpha^{-1})$. In particular, the bound holds for all $t\ge\alpha^{-2}$.
\end{proof}

We now show that a random $\ell$-walk on $G$ satisfies a similar rapid mixing property.

\begin{cor}
Suppose $1/n\ll\alpha\ll\hat\delta,1/k,1/b$. Let $\ell\in [k-1]_0$. Let $G$ be an $(n,k,\delta)$-graph and $\bx$ be a $b$-normal perfect fractional matching of $G$. Let $(R_t)_{t\in\bN_0}$ be a random $\ell$-walk on $G$ induced by $\bx$. Then, for any vertex $v\in V(G)$ and $t\ge \alpha^{-2}$, we have
$$\pr[R_t=v]=\left(1\pm e^{-\alpha t}\right)\frac{1}{n}.$$
\end{cor}

\begin{proof}
Let $Z=(\vec Z_t)_{t\in \bN_0}$ be the corresponding random walk on $\mathbb{S}_{\ell}[G]$ induced by $\bx$ and choose $\alpha$ as in Lemma~\ref{cor: rapid mixing in S[G] - non-tight}. Fix $t\ge \alpha^{-2}(k-\ell)$. Let $\tau\in\mathbb{N}_0$ and $\tau'\in[k]$ be indices such that $\vec Z_{\tau}=(s_1,...,s_k)$ is the first step in $Z$ that contains $R_t$ and $R_t=s_{\tau'}$. Note that $\tau\ge \alpha^{-2}$. So for all $v\in V(G)$, we have
\begin{align*}
\pr[R_t=v] &=\pr[\vec Z_{\tau}\in\{\vec e\in\vec E(G)|e_{\tau'}=v\}]=\sum_{\vec e\in\vec E(G)\colon e_{\tau'}=v}\pr[\vec Z_{\tau}=\vec e]\\
&=(k-1)!\sum_{v\in e\in E(G)}(1\pm e^{-\alpha t})\frac{\bx_e}{(k-1)!n}=(1\pm e^{-\alpha t})\frac{1}{n}
\end{align*}
which completes the proof.
\end{proof}

\subsection{Random walks and vertex subsets}
Recall the definition of well-behavedness at the end of Section~\ref{section: Probabilistic properties of random walks}, we now extend it to $\ell$-walks.

Let $G$ be an $(n,k,\delta)$-graph and $\bx$ be a perfect fractional matching of $G$. Let $0<\alpha<1/3$. Let $\ell\in [k-1]_0$. Let $Z=(\vec Z_0,...,\vec Z_m)$ be a random walk on $\mathbb{S}_{\ell}[G]$ induced by $\bx$. We say that $Z$ is \emph{self-avoiding} if $Z_i\setminus Z_{i-1}\not\subseteq\cup_{j=0}^{i-1}Z_j$ for all $i\in [m]$. That is, the corresponding $\ell$-walk in $G$ is an $\ell$-path. Let $R=(R_1,...,R_{m'+\ell})$ be the random $\ell$-walk on $G$ corresponding to the first $m'+\ell$ vertices in $Z$, where $m_i'\coloneqq m(k-\ell)$. We say that $Z$ is \emph{$(\mathbb{S}_{\ell}[G], \alpha)$-well-behaved} if $Z$ is self-avoiding and $R$ is $(G, \alpha)$-well-behaved. In this case, we also say that $R$ is a \emph{$(G, \alpha)$-well-behaved $\ell$-path}.

Using a union bound, it is straightforward to show that with high probability, a random $\ell$-walk is self-avoiding.

\begin{lemma}\label{lemma: a random walk is likely self-avoiding - non-tight}
Let $G$ be an $(n,k,\delta)$-graph and $\bx$ be a $b$-normal perfect fractional matching of $G$ for some $b\ge 1$. Let $\ell\in [k-1]_0$, $m\ge 0$ and $Z=(\vec Z_0,...,\vec Z_m)$ be a random walk on $\mathbb{S}_{\ell}[G]$ induced by $\bx$. Then
$$\pr[Z\text{ is self-avoiding}]\ge 1-\frac{2m^2k^{k+1}b^2}{n}.$$
\end{lemma}

We note that for $\ell\in[k-1]_0$, Corollary~\ref{cor: a random walk is likely (U,a,a')-expected} generalizes directly to random $\ell$-walks $(R_t)_{t\in\mathbb{N}_0}$ with $m\coloneqq \lfloor n^{1/3}\rfloor (k-\ell)$. All together, we show that with high probability, a random $\ell$-walk is $(\mathbb{S}_{\ell}[G],\alpha)$-well-behaved.

\begin{cor}\label{cor: a random walk is likely good - non-tight}
Suppose $1/n\ll\alpha\ll\hat\delta, 1/k, 1/b$. Let $\ell\in [k-1]_0$. Let $G$ be an $(n,k,\delta)$-graph and $\bx$ be a $b$-normal perfect fractional matching of $G$. Let $Z=(\vec Z_0,...,\vec Z_m)$ be a random walk on $\mathbb{S}_{\ell}[G]$ induced by $\bx$, where $m\coloneqq\lfloor n^{1/3}\rfloor$. Let $\cE$ be the event that $Z$ is $(\mathbb{S}_{\ell}[G],\alpha)$-well-behaved. Then \mbox{$\pr[\cE]\ge 1-n^{-1/4}$}.
\end{cor}

\begin{proof}
Using the generalization of Corollary~\ref{cor: a random walk is likely (U,a,a')-expected} and Lemma~\ref{lemma: a random walk is likely self-avoiding - non-tight}, we have
$$\pr[\cE]\ge 1-\frac{2m^2k^{k+1}b^2}{n}-3|\cS_{k-1}(G)|\frac{1}{n^k}\ge 1-2k^{k+1}b^2n^{-\frac{1}{3}}-3{n\choose k-1}\frac{1}{n^k}\ge1-n^{-\frac{1}{4}},$$
which completes the proof.
\end{proof}

\section{Counting Hamilton $\ell$-cycles}\label{section: Counting Hamilton ell-cycles}

\subsection{Counting well-behaved random walks}
We now find a lower bound on the entropy of a random walk on $\mathbb{S}_{\ell}[G]$. This then provides a lower bound on the number of such random walks. For convenience, given an $(n,k,\delta)$-graph $G$ and $\ell\in[k-1]_0$, we define
$$\overline h_{\ell}(G)\coloneqq h(G)-\frac{n}{k}\log {n\choose k-1}\frac{(n-k+1)!}{(n-\ell)!n}.$$
Note that \mbox{$\overline h_{\ell}(G)=\Theta(n\log n)$}.

\begin{lemma}\label{lemma: entropy of a random walk - non-tight}
Suppose $1/n\ll \hat\delta,1/k, 1/b,\eps$. Let $\ell\in[k-1]_0$. Let $G$ be an $(n,k,\delta)$-graph and $\bx$ be a $b$-normal perfect fractional matching of $G$ such that $h(\bx)\ge h(G)-\eps n$. Fix $\vec f\in \vec E(G)$. Let $Z=(\vec Z_0,...,\vec Z_m)$ be a random walk on $\mathbb{S}_{\ell}[G]$ induced by $\bx$ such that $\vec Z_0=\vec f$ and $m\coloneqq\lfloor n^{1/3}\rfloor$. Then $H(Z)\ge \frac{m}{n}k\overline h_{\ell}(G)-2k\eps m$.
\end{lemma}

\begin{proof}
We use the chain rule and Corollary~\ref{cor: rapid mixing in S[G] - non-tight} to obtain the desired result. Choose $\alpha$ with $1/n\ll \alpha\ll\hat\delta,1/k, 1/b$ such that Corollary~\ref{cor: rapid mixing in S[G] - non-tight} holds. Let $a\coloneqq\ln n+1>\alpha^{-2}$. Using the definition of a Markov chain, for $a\le i\le m$, we have
\begin{align*}
H(\vec Z_i|\vec Z_0,...,\vec Z_{i-1})&=H(\vec Z_i|\vec Z_{i-1})=\sum_{\vec e\in \vec E(G)}\pr[\vec Z_{i-1}=\vec e]H(\vec Z_i|{\vec Z_{i-1}=\vec e})\\
&\ge\sum_{\vec e\in \vec E(G)}\left(1-e^{-\alpha\ln n}\right)\frac{\bx_e}{(k-1)!n}h_{\bx}(e^{(\ell)+})\\
&\ge\left(1-n^{-\alpha}\right)\frac{1}{n}\frac{(k-\ell)!\ell!}{(k-1)!}\sum_{S\in \cS_{\ell}(G)}\bx_Sh_{\bx}(S).
\end{align*}
Using Lemma~\ref{lemma: S[G] property vi - non-tight}, we have
$$\sum_{S\in \cS_{\ell}(G)}\bx_Sh_{\bx}(S)\ge{k\choose\ell}h(\bx)-{k\choose\ell}\frac{n}{k}\log {n\choose k-1}\frac{(n-k+1)!}{(n-\ell)!n}\ge{k\choose\ell}\overline h_{\ell}(G)-{k\choose \ell}\eps n.$$
Using the chain rule of entropy (Fact~\ref{fact: CR for entropy}), we have
\begin{align*}
H(Z)&=\sum_{i=1}^mH(\vec Z_i|\vec Z_0,...,\vec Z
_{i-1})\ge (m-a+1)\left(1-n^{-\alpha}\right)\frac{1}{n}\left(k\overline h_{\ell}(G)-k\eps n\right)\\
&\ge\left(1-2n^{-\frac{1}{3}}\ln n-n^{-\alpha}\right)\frac{m}{n}\left(k\overline h_{\ell}(G)-k\eps n\right)\ge \frac{m}{n}k\overline h_{\ell}(G)-2k\eps m,
\end{align*}
which completes the proof.
\end{proof}

\begin{cor}\label{cor: counting well-behaved paths - non-tight}
Suppose $1/n\ll\alpha\ll \hat\delta,1/k, 1/b,\eps $. Let $\ell\in[k-1]_0$. Let $G$ be an $(n,k,\delta)$-graph and $\bx$ be a $b$-normal perfect fractional matching of $G$ such that $h(\bx)\ge h(G)-\eps n$. Fix $\vec e\in \vec E(G)$. Let $Z=(\vec Z_0,...,\vec Z_m)$ be a random walk on $\mathbb{S}_{\ell}[G]$ induced by $\bx$ such that $\vec Z_0=\vec e$ and $m\coloneqq\lfloor n^{1/3}\rfloor$. Let $\cZ$ be the set of $(\mathbb{S}_{\ell}[G], \alpha)$-well-behaved realizations of $Z$. Then
$\log |\cZ|\ge \frac{m}{n}k\overline h_{\ell}(G)-3k\eps m$.
\end{cor}

\begin{proof}
Choose $\alpha$ with $1/n\ll\alpha\ll\hat\delta,1/k,1/b$ such that Corollary~\ref{cor: a random walk is likely good - non-tight} holds. Let $\cE$ denote the event that $Z$ is $(\mathbb{S}_{\ell}[G], \alpha)$-well-behaved. By Corollary \ref{cor: a random walk is likely good - non-tight}, we have $\pr[\cE]\ge 1-n^{-1/4}$. Let $W$ be any realization of $Z$, then Lemma~\ref{lemma: S[G] property v - non-tight} yields $\pr[Z=W]\ge ((k-\ell)!b^2n^{k-\ell})^{-m}.$ By Facts~\ref{fact: max entropy} and~\ref{fact: entropy given a likely event} as well as Lemma~\ref{lemma: entropy of a random walk - non-tight}, we have
$$\log|\cZ|\ge H(Z|\cE)\ge H(Z)-2n^{-\frac{1}{4}}\log ((k-\ell)!b^2n^{k-\ell})^m\ge \frac{m}{n}k\overline h_{\ell}(G)-3k\eps m,$$
as desired.
\end{proof}

\subsection{Counting long paths}
The next lemma generalizes Lemma~\ref{lemma: graph remains Dirac after removal of a well-behaved path} and forms the basis for iterating Corollary~\ref{cor: counting well-behaved paths - non-tight}. The proof is the same as the special case of $\ell=k-1$ (Lemma~\ref{lemma: graph remains Dirac after removal of a well-behaved path}). Recall that Lemma~\ref{lemma: entropy is as expected after removing a well-behaved subset} still holds when $m\coloneqq\lfloor n^{1/3}\rfloor(k-\ell)$, for any $\ell\in[k-1]_0$.

\begin{lemma}\label{lemma: graph remains Dirac after removal of a well-behaved path - non-tight}
Suppose $1/n\ll\alpha\ll 1/b\ll\eps \ll\hat\delta, 1/k$. Let $\ell\in[k-1]_0$. Let $G$ be an $(n,k,\delta)$-graph and $\bx$ be a $b$-normal perfect fractional matching of $G$ such that $h(\bx)\ge h(G)-\eps n$. Suppose $P=v_1...v_m$ is a $(G, \alpha)$-well-behaved $\ell$-path, where $m\coloneqq\lfloor n^{1/3}\rfloor(k-\ell)$. Let $G'\coloneqq G-P$, $n'\coloneqq n-m$, $\delta'\coloneqq \delta -2n^{-\frac{2}{3}-\alpha}$ and $b'\coloneqq (1+n^{-\frac{2}{3}-\frac{\alpha}{2}})b$. Then $G'$ is an $(n', k, \delta')$-graph, and there exists a $b'$-normal perfect fractional matching~$\bx'$ of $G'$ such that $h(\bx')\ge h(G')-\eps n'$ and
\begin{equation}
h(\bx')\ge \frac{n'}{n}h(\bx)-\frac{k-1}{k}n'\log \frac{n}{n'}-n^{\frac{1}{3}-\frac{\alpha}{2}}.
\end{equation}
\end{lemma}

\begin{lemma}\label{lemma: counting long paths - non-tight}
Suppose $1/n\ll \beta\ll\eps\ll\hat\delta, 1/k$. Let $\ell\in[k-1]_0$. Let $G$ be an $(n,k,\delta)$-graph. Fix $\vec e\in\vec E(G)$. Let $z$ be the number of ordered $\ell$-paths in $G$ that start at $\vec e$ and have length at least $\frac{n-n^{1-\beta}}{k-\ell}$. Then $\log z\ge \frac{k}{k-\ell}\overline h(G)-n\log e-\eps n.$
\end{lemma}

\begin{proof}
Let $L\coloneqq\frac{n-n^{1-\beta}}{k-\ell}$. Choose $\eps$ and $b\ge 4$ such that Lemma~\ref{lemma: existence of b-normal pfm} holds for the parameters $\frac{\hat\delta}{2}$ (instead of $\hat\delta$) and $k$. Choose $\alpha\coloneqq 3\beta$ so that Lemma~\ref{cor: counting well-behaved paths - non-tight} and Lemma~\ref{lemma: graph remains Dirac after removal of a well-behaved path - non-tight} hold for the parameters $\frac{\hat\delta}{2}$ (instead of $\hat\delta$), $k,\eps $ and $2b$ (instead of $b$). By Lemma~\ref{lemma: existence of b-normal pfm}, $G$ admits a $b$-normal perfect fractional matching $\bx$ with $h(\bx)\ge h(G)-\eps n$. We claim that $\delta_i\ge\frac{1}{2}+\frac{\hat\delta}{2}$ and $b_i\le 2b$ in the following procedure. Let $G_0\coloneqq G$, $n_0\coloneqq n$, $\delta_0\coloneqq\delta$, $b_0\coloneqq b$ $\bx_0\coloneqq \bx$ and $\vec e_0\coloneqq \vec e$. We construct an ordered $\ell$-path in $G$ as follows.

At each step $i\ge 0$, let $\cZ_{i}$ be the set of all $(\mathbb{S}_{\ell}[G_i],\alpha)$-well-behaved realizations of the random walk $Z=(\vec Z_0,...,\vec Z_{m_i})$ on $\mathbb{S}_{\ell}[G_i]$ induced by $\bx_i$, where $\vec Z_0=\vec e_i$ and $m_i\coloneqq \lfloor n_i^{1/3}\rfloor$. Fix some walk $Z_i\in \cZ_i$, and let $V_i$ be the set of the first $m_i'\coloneqq m_i(k-\ell)$ vertices in $Z_i$. Let $P_i$ be the ordered $\ell$-path in $G_i$ consisting of the first $m_i'+\ell$ vertices in $Z_i$, then $P_i$ is a $(G_i,\alpha)$-well-behaved $\ell$-path. Let $\vec e_{i+1}$ be the end of $Z_i$. Define
\begin{align*}
G_{i+1}&\coloneqq G_i-V_i,\\
n_{i+1}&\coloneqq n_i-m_i',\\
\delta_{i+1}&\coloneqq \delta_i-2n_i^{-\frac{2}{3}-\alpha},\\
b_{i+1}&\coloneqq \left(1+n_i^{-\frac{2}{3}-\frac{\alpha}{2}}\right)b_i,\text{ and}\\
z_i&\coloneqq \min|\cZ_i|,
\end{align*}
where the minimum is taken over all possible ordered $\ell$-paths $P_0\cup...\cup P_{i-1}$ and starting points $\vec e_i$. By Lemma~\ref{lemma: graph remains Dirac after removal of a well-behaved path - non-tight}, $G_{i+1}$ is an $(n_{i+1},k,\delta_{i+1})$-graph, and there exists a $b_{i+1}$-normal perfect fractional matching $\bx_{i+1}$ of $G_{i+1}$ such that $h(\bx_{i+1})\ge h(G_{i+1})-\eps n_{i+1}$ and 
\begin{equation}
h(\bx_{i+1})\ge\frac{n_{i+1}}{n_i}h(\bx_i)-\frac{(k-1)}{k}n_{i+1}\log \frac{n_i}{n_{i+1}}-n_i^{\frac{1}{3}-\frac{\alpha}{2}}.
\end{equation}

We terminate the procedure once $n_i\le n^{1-\beta}.$ Let $\kappa$ be the step such that $n_{\kappa+1}\le n^{1-\beta}<n_{\kappa}$. Note that $Q\coloneqq P_0\cup...\cup P_{\kappa}$ is an ordered $\ell$-path of length at least $L$. This procedure is finite since at every step $i$, $G_{i+1}$ is a proper subgraph of $G_i$. Moreover, since $m_0'\ge...\ge m_{\kappa-1}'$ and $m_0'+...+m_{\kappa-1}'\le n$, we have
$$\kappa\le \frac{n}{m_{\kappa-1}'}\le \frac{n}{(k-\ell)n_{\kappa-1}^{1/3}}<\frac{1}{k-\ell}\cdot\frac{n}{(n^{1-\beta})^{\frac{1}{3}}}\le \frac{1}{k-\ell}\cdot n^{\frac{2}{3}+\frac{\beta}{3}}.$$
By similar arguments as~(\ref{eq: delta_i bound}) and~(\ref{eq: b_i bound}), for all $i\in [\kappa]_0$, we have $\delta_i\ge\frac{1}{2}+\frac{\hat\delta}{2}>\frac{1}{2}$ and $b_i\le 2b$. Thus, $G_i$ is an $(n_i, k,\frac{1}{2}+\frac{\hat\delta}{2})$-graph and $\bx_i$ is a $2b$-normal perfect fractional matching of $G_i$. Again by induction on $i\in [\kappa]_0$, we obtain the same inequality as (\ref{eq: induction}), which is
$$h(G_i)\ge h(\bx_i)\ge \frac{n_i}{n}h(\bx)-\frac{k-1}{k}n_i\log \frac{n}{n_i}-in^{\frac{1}{3}-\frac{\alpha}{2}}.$$
By Corollary~\ref{cor: counting well-behaved paths - non-tight}, we have
\begin{align*}
\log z_i&\ge k\frac{m_i}{n_i}h(G_i)-m_i\log {n_i\choose k-1}\frac{(n_i-k+1)!}{(n_i-\ell)!n}-3k\eps m_i\\
&\ge k\frac{m_i}{n}h(\bx)-(k-1)m_i\log \frac{n}{n_i}-ik\frac{m_i}{n_i}n^{\frac{1}{3}-\frac{\alpha}{2}}-m_i\log {n_i\choose k-1}\frac{(n_i-k+1)!}{(n_i-\ell)!n_i}-3k\eps m_i\\
&\ge k\frac{m_i}{n}h(\bx)-m_i\log{n_i\choose k-1}\frac{(n_i-k+1)!}{(n_i-\ell)!}\left(\frac{n}{n_i}\right)^{\ell}-(k-\ell-1)m_i\log n+(k-\ell)m_i\log n_i\\
&\quad-ik\frac{m_i}{n_i}n^{\frac{1}{3}-\frac{\alpha}{2}}-3k\eps m_i.
\end{align*}
We consider the following sums over $i\in[\kappa]_0$. By construction, $\sum_{i=0}^{\kappa}m_i'=n-n_{\kappa+1}\ge n-n^{1-\beta}$. Similar to~(\ref{eq: sum 1}) and~(\ref{eq: sum 2}), we have
$$\sum_{i=0}^{\kappa}m_i'\log n_i\ge n\log n-n\log e-\frac{\eps}{2}n\quad\text{and}\quad\sum_{i=0}^{\kappa}i\frac{m_i}{n_i}n^{\frac{1}{3}-\frac{\alpha}{2}}\le\frac{\eps}{2}n.$$
Note that for $i\in[\kappa]$, we have
\begin{align*}
{n_i\choose k-1}\frac{(n_i-k+1)!}{(n_i-\ell)!}\left(\frac{n}{n_i}\right)^{\ell}&={n\choose k-1}\frac{n_i!(n-k+1)!}{n!(n_i-k+1)!}\cdot\frac{(n_i-k+1)!}{(n_i-\ell)!}\left(\frac{n}{n_i}\right)^{\ell}\\
&={n\choose k-1}\frac{(n-k+1)!}{(n-\ell)!}\cdot\frac{n_i!(n-\ell)!}{n!(n_i-\ell)!}\left(\frac{n}{n_i}\right)^{\ell}\\
&\le{n\choose k-1}\frac{(n-k+1)!}{(n-\ell)!}.
\end{align*}
Recall that $m_i'=m_i(k-\ell)$ for $i\in[\kappa]_0$. Let $z$ denote the number of ordered $\ell$-paths that start at $\vec e$ and have length at least $L$. Note that every $\ell$-path $Q$ constructed above is such a path. Thus, we have
\begin{align*}
\log z&\ge \log \prod_{i=0}^{\kappa} z_i=\sum_{i=0}^{\kappa}\log z_i\\
&\ge\sum_{i=0}^{\kappa} \frac{k}{k-\ell}\frac{m_i'}{n}h(\bx)-\frac{m_i'}{k-\ell}\log{n\choose k-1}\frac{(n-k+1)!}{(n-\ell)!}-\frac{k-\ell-1}{k-\ell}m_i'\log n+m_i'\log n_i\\
&\quad-ik\frac{m_i}{n_i}n^{\frac{1}{3}-\frac{\alpha}{2}}-3k\eps m_i\\
&\ge \frac{k}{k-\ell}h(G)-\frac{n}{k-\ell}\log{n\choose k-1}\frac{(n-k+1)!}{(n-\ell)!}-\frac{k-\ell-1}{k-\ell}n\log n+n\log \frac{n}{e}-2\eps n-3k\eps n\\
&\ge\frac{k}{k-\ell}\overline h(G)-n\log e-5k\eps n.
\end{align*}
This completes the proof.
\end{proof}

\subsection{Absorbing long paths}
Given an $(n,k,\delta)$-graph $G$, we again use Lemma~\ref{lemma: absorption subset} to pick a vertex subset $U\subseteq V(G)$ that contains sufficiently many neighbours of every $(k-1)$-subset of $V(G)$. Applying Lemma~\ref{lemma: counting long paths - non-tight}, we find many long ordered $\ell$-paths in $G-U$ and use Lemma~\ref{lemma: tight Hamilton-connected} to complete each of these paths into an ordered Hamilton $\ell$-cycle of $G$. We are now able to prove our main theorem.

\begin{proof}[Proof of Theorem~\ref{main theorem - with entropy - non-tight}]
Choose $\beta$ and $\eps$ such that Lemma~\ref{lemma: counting long paths} holds for the parameters $\frac{\hat\delta}{2}$ (instead of $\hat\delta$) and $k$. Let $U\subseteq V(G)$ be a set of size $n^{\alpha}$ guaranteed by Lemma~\ref{lemma: absorption subset} with $\alpha\coloneqq 1-\beta/2$. Let $G'\coloneqq G-U$ and $n'\coloneqq n-n^{\alpha}$. Then $G'$ is an $(n', k, \frac{1}{2}+\frac{\hat\delta}{2})$-graph since
$$\delta_{k-1}(G')\ge \delta_{k-1}(G)-|U|\ge\delta n-n^{\alpha}\ge\left(\frac{1}{2}+\frac{\hat\delta}{2}\right)n'.$$
By Lemma~\ref{lemma: entropy decreases slightly after removing a small subset}, we have $h(G')\ge h(G)-\eps n$. Fix $\vec e\in \vec E(G')$. Let $Q$ be an ordered $\ell$-path in $G'$ of length at least $\frac{n'-(n')^{1-\beta}}{k-\ell}$ from $\vec e$ to $\vec f$ for some $\vec f\in\vec E(G')$. Let $\vec S\coloneqq \vec f^{(k-1)+}$, $\vec T\coloneqq \vec e^{(k-1)-}$, $G''\coloneqq G[V(G'-Q)\cup U\cup S \cup T]$, and $n''\coloneqq |V(G'')|\le (n')^{1-\beta}+n^{\alpha}+2k\le n^{\alpha}+n^{1-\beta}$. Note that $G''$ is an $(n'', k, \frac{1}{2}+\frac{\hat\delta}{2})$-graph since for all $S\in \cS_{k-1}(G'')\subseteq \cS_{k-1}(G)$, we have
$$|N_{G''}(S)|\ge |N_G(S)\cap U|\ge \left(\frac{1}{2}+\frac{3\hat\delta}{4}\right)n^{\alpha}\ge\left(\frac{1}{2}+\frac{\hat\delta}{2}\right)(n^{\alpha}+n^{1-\beta})\ge\left(\frac{1}{2}+\frac{\hat\delta}{2}\right)n''.$$
By Lemma~\ref{lemma: tight Hamilton-connected}, there is a tight Hamilton path $P$ from $\vec S$ to $\vec T$ in $G''$. Observe that $P$ contains an ordered Hamilton $\ell$-path from $\vec f^{(\ell)+}$ to $\vec e^{(\ell)-}$. Then $P+Q$ is an ordered Hamilton $\ell$-cycle of $G$. Since $Q$ was arbitrary, by Lemma~\ref{lemma: counting long paths}, we have
\begin{align*}
\log \Psi_{\ell}'(G)&\ge \frac{k}{k-\ell}h(G')-\frac{n'}{k-\ell}\log {n'\choose k-1}\frac{(n'-k+1)!}{(n'-\ell)!n'}-n'\log e-\eps n'\\
&\ge \frac{k}{k-\ell}(h(G)-\eps n)-\frac{n}{k-\ell}\log {n\choose k-1}\frac{(n-k+1)!}{(n-\ell)!n}-n\log e-2\eps n\\
&\ge \frac{k}{k-\ell}h(G)-\frac{n}{k-\ell}\log {n\choose k-1}\frac{(n-k+1)!}{(n-\ell)!n}-n\log e-3k\eps n.
\end{align*}
Recall that $\Psi_{\ell}'(G)=C_{n,k,\ell}\Psi_{\ell}(G)$. This completes the proof.
\end{proof}

When $\ell=0$, we have $\log \Phi(G)\ge h(G)-\frac{k-1}{k}n\log e-o(n)$, which is the lower bound in Theorem~\ref{thm: KSW}. Theorem~\ref{thm: FHM} also follows as a corollary.
\begin{proof}[Proof of Theorem~\ref{thm: FHM}]
Using the fact that $-\log (1-x)=\log(1+\frac{x}{1-x})\le \frac{2x}{1-x}\le 3x$ for $0< x\le 1/3$, we first calculate
$$\frac{1}{k-\ell}\log\frac{ n(n-\ell)!}{(n-k+1)!}\ge\frac{1}{k-\ell}\log (n-k+2)^{k-\ell}\ge \log n\left(1-\frac{k-2}{n}\right)\ge\log n-\frac{3(k-2)}{n}.$$
We use Theorem~\ref{main theorem - with entropy - non-tight}, together with the lower bound for $h(G)$ given by Theorem~\ref{thm: graph entropy bound} to obtain
\begin{align*}
\log \Psi_{\ell}'(G)&\ge \frac{n}{k-\ell}\log{n\choose k-1}+\frac{n}{k-\ell}\log \delta-\frac{n}{k-\ell}\log{n\choose k-1}\frac{(n-k+1)!}{(n-\ell)!n}-n\log e-o(n)\\
&\ge\frac{n}{k-\ell}\log\delta+\frac{n}{k-\ell}\log\frac{(n-\ell)!n}{(n-k+1)!}-n\log e-o(n)\\
&\ge\frac{n}{k-\ell}\log\delta+n\log n-3(k-2)-n\log e-o(n)\\
&\ge\frac{n}{k-\ell}\log\delta +\log (n!)-o(n).
\end{align*}
Therefore, we conclude $\Psi_{\ell}(G)\ge C_{n,k,\ell}^{-1}n!(\delta-o(1))^{\frac{n}{k-\ell}}=\Psi_{\ell}(K_n^{(k)})(\delta-o(1))^{\frac{n}{k-\ell}}$.
\end{proof}

\end{document}